\documentclass{article}
\usepackage{hyperref}
\usepackage[american]{babel}
\usepackage{amsfonts,amsmath,amssymb,epsf,epsfig, stmaryrd}

\usepackage{amsthm}

\usepackage{xypic}
\xyoption{all}
\input{xypic}
\newdir{ >}{{}*!/-8pt/\dir{>}}

\newcommand{\id}{{\rm id}}

\renewcommand{\hat}{\widehat} 
\newcommand{\im}{\mathop{{\rm im}}\nolimits}
\newcommand{\Hom}{\mathop{{\rm Hom}}\nolimits}

\def\R{\mathbb{R}}

     \def\Hom{{\rm Hom}}
\def\ad{{\rm ad}}  
\def\pr{$\bf{Proof.}$\quad}
\def\fin{\hfill$\square$\\}

\newtheorem{theo}{Theorem}[section]
\newtheorem{defi}[theo]{Definition}
\newtheorem{rema}[theo]{Remark}
\newtheorem{prop}[theo]{Proposition}
\newtheorem{cor}[theo]{Corollary}
\newtheorem{lem}[theo]{Lemma}
\newtheorem{pro}[theo]{Problem}

\newenvironment{rem}{\begin{rema}\rm}{\end{rema}}

\begin{document}
\title{Deformation quantization of Leibniz algebras} 

\author{Benoit Dherin\\
University of California, Berkeley, USA\\
dherin@math.berkeley.edu \\
\and Friedrich Wagemann\\
     Universit\'e de Nantes, France\\
wagemann@math.univ-nantes.fr}

\maketitle

\begin{abstract}
This paper has two parts. The first part is a review and extension of the 
methods of integration of Leibniz algebras into Lie racks, including as new feature
a new way of integrating $2$-cocycles (see Lemma 3.9). 

In the second part, we use the local integration of a Leibniz algebra ${\mathfrak h}$ 
using a Baker-Campbell-Hausdorff type formula in order to deformation quantize its linear dual
${\mathfrak h}^*$. More precisely, we define a natural rack product on the set of 
exponential functions
which extends to a rack action on ${\mathcal C}^{\infty}({\mathfrak h}^*)$.       
\end{abstract}

\section*{Introduction}

In this paper, we solve an old problem in symplectic geometry, namely we propose a way how to 
quantize the dual space of a Leibniz algebra ${\mathfrak h}$. This dual space ${\mathfrak h}^*$
is some kind of generalized Poisson manifold, as the bracket of ${\mathfrak h}$ is not necessarily
skew-symmetric. Intimately linked to this question is the integration of Leibniz algebras.   

One of the most fascinating theorems in Lie theory is Lie's Third Theorem, namely the possibility
to integrate every real Lie algebra into a Lie group. Several proofs of this theorem are known,
none of them reduces the claim to easy facts or computations. We focus here especially on two 
approaches. The first one, which we call the homological proof of Lie's Third Theorem, regards
a given Lie algebra ${\mathfrak g}$ as a central extension of its adjoint Lie algebra 
${\mathfrak g}_{\rm ad}$ by its center $Z({\mathfrak g})$ and uses then the fact that 
${\mathfrak g}_{\rm ad}$ is embedded by the adjoint action as a subalgebra of 
${\mathfrak g}{\mathfrak l}({\mathfrak g})$, thus integrating by Lie's First Theorem into a 
Lie subgroup
of ${\rm Gl}({\mathfrak g})$. It remains then to integrate the $2$-cocycle determining the 
central extension,
therefore we call it {\it homological proof}. Another approach, which we call the approach using Ado's 
Theorem, uses the fact that every real Lie algebra embeds as a subalgebra of some matrix Lie algebra
(Ado's Theorem) to integrate once again only subalgebras of general linear Lie algebras into 
Lie subgroups of general linear groups. 

In the search of understanding the periodicity in K-theory, J.-L. Loday introduced Leibniz algebras 
as non-commutative analogues of Lie algebras. More precisely, a real Leibniz algebra is a real
vector space with a bracket which satisfies the (left) Leibniz identity 
$$[X,[Y,Z]]\,=\,[[X,Y],Z]+[Y,[X,Z]],$$
but is not necessarily skew-symmetric. Leibniz algebras are a well-established algebraic structure
generalizing Lie algebras (those Leibniz algebras where the bracket is skew-symmetric) with their own 
structure-, deformation- and homology theory. 
In the same way the Lie algebra homology of matrices (over a commutative ring containing the rational 
numbers) defines additive K-theory (i.e. cyclic homology), the Leibniz homology
of matrices defines some non-commutative additive K-theory (in fact, Hochschild homology). 
Loday was mainly interested in the 
properties of the corresponding homology theory on ``group level'' (``Leibniz K-Theory''), 
and therefore asked the question which (generalization of the 
structure of Lie groups) is the correct structure to integrate Leibniz algebras ?

This paper consists of two parts.
The main goal of the first part (comprising Sections 1 to 3) is to compare the integration 
procedures for Leibniz algebras
which one may generalize from the homological proof and from the proof using Ado's Theorem of Lie's
Third Theorem. 
Kinyon \cite{Kin} explored Lie racks as a structure integrating Leibniz algebras. 
Racks are roughly speaking an axiomatization of the structure of the conjugation in a group.
The rack product on a group is simply given by
$$g\rhd h\,:=\,ghg^{-1},$$
and a general rack product on a set $X$ is an invertible binary operation satisfying for all $x,y,z\in X$ 
the autodistributivity relation
$$x\rhd(y\rhd z)\,=\,(x\rhd y)\rhd(x\rhd z).$$
Lie racks are the smooth analogue of racks. Kinyon showed (see Theorem \ref{Kinyon1}) that the tangent 
space at the distiguished 
element $1$ of a Lie rack carries in a natural way a Leibniz bracket. The idea is to differentiate
two times the rack structure, mimicking exactly how the conjugation in a Lie group is differentiated 
to give first the map ${\rm Ad}$, the adjoint action of the group on the Lie algebra, and then the 
Lie bracket in terms of ${\rm ad}$, the adjoint action of the Lie algebra on itself. He did not
see racks as the correct objects integrating Leibniz algebras. As a reason for this, he showed 
that all Leibniz algebras integrate into Lie racks, but in a kind of arbitrary way, as this integration
does not appear to give Lie groups in case one started with a Lie algebra. He more or less exhibited
Ado's approach to the integration of Leibniz algebras into (what he called) linear Lie racks.
We explain this in Section $3$. It is clear (and useful as a guiding principle) that from this point of
view, integrating Leibniz algebras means just an {\it integration of the adjoint action} of a Leibniz 
algebra on itself. From here stems the most important example of a rack product, namely 
$$X\rhd Y\,:=\,e^{{\rm ad}_X}(Y),$$
for all $X,Y\in{\mathfrak h}$ for a Leibniz algebra ${\mathfrak h}$. 

On the other hand, Covez \cite{Cov} showed in his 2010 doctoral thesis how to adapt the 
homological proof of 
Lie's Third Theorem to Leibniz algebras. Regarding a given real Leibniz algebra ${\mathfrak h}$
as an abelian extension of the Lie algebra ${\mathfrak h}_{\rm Lie}$ by its left center 
$Z_L({\mathfrak h})$, he integrated Leibniz algebras into local Lie racks. The fact that this procedure
works only locally stems from the fact that the Leibniz $2$-cocycle governing the abelian extension 
is only integrated into a local rack cocycle, due to the use of open sets on the Lie group $G_0$
integrating ${\mathfrak h}_{\rm Lie}$ where exponential and logarithm are mutually inverse 
diffeomorphisms. This integration has the advantage of specializing to the conjugation racks 
associated to Lie groups in case the given Leibniz algebra is a Lie algebra. All this is 
explained in some detail in Section $2$.  

This first part of this paper clarifies the integration of Leibniz algebras and shows to our 
belief that Lie racks 
are indeed the correct structure for this integration. Along the way, we show that in all of our
integration procedures the integrating object reduces locally to (the conjugation rack of) a Lie 
group in case we are dealing with a Lie algebra. The comparison of these approaches is summarized in 
Section \ref{summary_section}. 

Other evidence that racks are the right objects integrating Leibniz algebras comes from recent 
work of Covez on product structures on rack homology showing that it has (some of) the expected
properties of a Leibniz K-theory which were predicted by Loday.

In the second part of our paper (Section 4), we use the integration procedure of Leibniz 
algebras set up in 
Section $3$ in order to develop deformation quantization of Leibniz algebras.

Given a finite-dimensional real Lie algebra $({\mathfrak g},[,])$, 
its dual vector space ${\mathfrak g}^*$ is a smooth manifold
which carries a Poisson bracket on its space of smooth functions, defined for all 
$f,g\in{\mathcal C}^{\infty}({\mathfrak g}^*)$ and all $\xi\in{\mathfrak g}^*$
by the {\it Kostant-Kirillov-Souriau formula}
$$\{f,g\}(\xi)\,:=\,\langle \xi,[df(\xi),dg(\xi)]\rangle.$$
Here $df(\xi)$ and $dg(\xi)$ are linear functionals on ${\mathfrak g}^*$, identified with elements of
${\mathfrak g}$.

In the same way, a general Leibniz algebra ${\mathfrak h}$ gives rise to a smooth manifold 
${\mathfrak h}^*$, which carries now some kind of generalized Poisson bracket, in
particular, the bracket need
not be skew-symmetric. We call manifolds with such a bracket {\it generalized Poisson manifolds}. 

It is well known that the deformation quantization of the Poisson manifold ${\mathfrak g}^*$ for a Lie 
algebra ${\mathfrak g}$ is intimitely related to the integration of the bracket of
${\mathfrak g}$ into a local/formal group product via the Baker-Campbell-Hausdorff (BCH) formula. 
The main idea of the present paper is to use the corresponding BCH-formula for the integration
of a Leibniz algebra ${\mathfrak h}$ in order to perform the corresponding deformation quantization.
  
The quantization technique we use relies on the quantization of special canonical relation germs,
called {\it symplectic micromorphisms} (see \cite{CDWI}, \cite{CDWII}, \cite{CDWIII}, and \cite{CDWIV}), 
by Fourier integral operators. 
In the Lie algebra case, we show that it is possible to re-interprete the Gutt star-product in terms 
of a symplectic micromorphism quantization, obtained by considering the cotangent lift of the local
group structure on the Lie algebra. We show that this quantization method also works for Leibniz algebras,
provided one takes the cotangent lift of the local rack structure for the symplectic micromorphism.
This local rack structure comes from the integration procedure exposed in the first part.   
 

The quantization of the dual of a Leibniz algebra ${\mathfrak h}^*$ that results from quantizing the 
symplectic micromorphism obtained from the local rack structure, is an operation 
$$\rhd:{\mathcal C}^{\infty}({\mathfrak h}^*)[[\epsilon]]\times
{\mathcal C}^{\infty}({\mathfrak h}^*)[[\epsilon]]\to
 {\mathcal C}^{\infty}({\mathfrak h}^*)[[\epsilon]]$$
such that the restriction of $\rhd$ to ``unitaries'' $U_{\mathfrak h}:=\{E_X\,|\,X\in{\mathfrak h}\}$ ($E_X$ being the 
exponential function on ${\mathfrak h}^*$ associated to $X\in{\mathfrak h}$) is a rack
structure $\rhd:U_{\mathfrak h}\times U_{\mathfrak h}\to U_{\mathfrak h}$. 

Moreover, the restriction of this operation to 
$$\rhd:U_{\mathfrak h}\times{\mathcal C}^{\infty}({\mathfrak h}^*)[[\epsilon]]\to
 {\mathcal C}^{\infty}({\mathfrak h}^*)[[\epsilon]]$$
should be a rack action. 

Our main theorem shows exactly this:\\

\noindent{\bf Theorem 4.12}{\it 
\quad The operation 
$$\rhd_{\hbar}:{\mathcal C}^{\infty}({\mathfrak h}^*)[[\epsilon]]\otimes
{\mathcal C}^{\infty}({\mathfrak h}^*)[[\epsilon]]\to
{\mathcal C}^{\infty}({\mathfrak h}^*)[[\epsilon]]$$
defined by 
$$f\rhd_{\hbar}g\,:=\,Q^{a=1}(T^*\rhd)(f\otimes g)$$
is a quantum rack, i.e.
\begin{enumerate}
\item $\rhd_{\hbar}$ restricted to $U_{\mathfrak h}=\{E_X\,|\,X\in{\mathfrak h}\}$ is a rack structure,
moreover
$$e^{\frac{i}{\hbar}X}\rhd_{\hbar}e^{\frac{i}{\hbar}Y}\,=\,e^{\frac{i}{\hbar}e^{{\rm ad}_X}(Y)},$$
\item $\rhd_{\hbar}$ restricted to 
$$\rhd_{\hbar}:U_{\mathfrak h}\times{\mathcal C}^{\infty}({\mathfrak h}^*)\to
{\mathcal C}^{\infty}({\mathfrak h}^*)$$ 
is a rack action;
$$(e^{\frac{i}{\hbar}X}\rhd_{\hbar}f)(\xi)\,=\,({\rm Ad}_{-X}^*f)(\xi).$$
\end{enumerate}
Moreover, $\rhd_{\hbar}$ coincides with the Gutt quantum rack $f\rhd_a g:=f*_ag*_a\overline{f}$
on unitaires in the Lie case (although it is different on the whole
${\mathcal C}^{\infty}({\mathfrak h}^*)[[\epsilon]]$).}\\

The quantization operator $Q^{a=1}$ occuring here is constructed as the Fourier Integral Operator 
associated to an amplitude $a$ and a generating function $S_{\rhd}$ according to 
$$Q^a(f\otimes g)(\xi)\,=\,\int_{{\mathfrak h}\times{\mathfrak h}}\widehat{f}(X)\widehat{g}(Y)
a(X,Y,\xi)e^{\frac{i}{\hbar}S_{\rhd}(X,Y,\xi)}\frac{dXdY}{(2\pi\hbar)^{n}},$$
where $f,g\in{\mathcal C}^{\infty}({\mathfrak h}^*)[[\epsilon]]$, $X,Y\in{\mathfrak h}$, 
$\xi\in{\mathfrak h}^*$
and $n=\dim({\mathfrak h})$. 
The generating function $S_{\rhd}$ relies on a Baker-Campbell-Hausdorff type formula for 
the Leibniz case, namely
$$S_{\rhd}(X,Y,\xi)\,:=\,\langle\xi, e^{{\rm ad}_X}(Y)\rangle.$$

The problem of deformation quantizing a Leibniz algebra has been addressed by other authors,
namely by K. Uchino in \cite{Uch} in the realm of associative dialgebras. 

Observe that thanks to this theorem, the general integration problem for generalized 
Poisson manifolds makes sense. 
Namely, given a generalized Poisson manifold, i.e. a manifold $M$ together with a 
bracket on ${\mathcal C}^{\infty}(M)$ satisfying similar properties
as the bracket on ${\mathcal C}^{\infty}({\mathfrak h}^*)$. Then 
we may ask whether there exists
a natural rack structure on the set of exponential functions which extends 
to a rack action on all smooth functions. 
Our main theorem solves this integration problem for {\it linear} generalized 
Poisson structures.  
 
A side result is a new way of integrating $2$-cocycles (see Lemma 3.9), which 
we find by comparing Covez' integration
and the BCH integration procedure.     

\vspace{1cm}

\noindent{\bf Acknowledgements:} FW is grateful to UC Berkeley for hospitality and 
excellent working conditions
during our work on this article. He thanks especially Alan Weinstein for the invitation, 
guidance and advice, and 
most useful discussions about the integration of Leibniz algebras. FW acknowledges 
support from CNRS during this period. FW thanks Yannick Voglaire for correcting the 
coadjoint action, and K. Uchino for correcting the notion of a generalized Poisson manifold
and bringing \cite{GraMar} to our attention.    

\section{Preliminaries on Leibniz and Lie algebras}

\subsection{Derivations of Leibniz algebras}

Fix a field $k$. Later we will specialize to $k=\R$ in order to speak about the exponential map
(although this is not mandatory). 
We present here a recollection of facts from Lie algebra theory which we generalize to Leibniz algebras
by showing that the usual proof still holds true in the Leibniz context. 

\begin{defi}
A (left) Leibniz algebra is a $k$-vector space ${\mathfrak h}$ together with a $k$-bilinear bracket
$[,]:{\mathfrak h}\times{\mathfrak h}\to{\mathfrak h}$ such that for all 
$X,Y,Z\in{\mathfrak h}$
$$[X,[Y,Z]]\,=\,[[X,Y],Z]+[Y,[X,Z]].$$
\end{defi}

\begin{defi}
A (left) derivation of a Leibniz algebra ${\mathfrak h}$ is a $k$-linear map 
$D:{\mathfrak h}\to{\mathfrak h}$
such that for all $X,Y\in{\mathfrak h}$
$$D([X,Y])\,=\,[D(X),Y]+[X,D(Y)].$$
\end{defi} 

Observe that the above left Leibniz identity means that for all $X\in{\mathfrak h}$,
${\rm ad}_X:=[X,-]$ is a (left) derivation of the bracket. 
Obviously, in case the bracket is 
also skew-symmetric, ${\mathfrak h}$ becomes a Lie algebra and the left Leibniz 
identity becomes the usual Jacobi identity. As this need not be 
the case, the notion of Leibniz algebra generalizes the notion of Lie algebra.
Observe furthermore that skew-symmetrizing the bracket of a Leibniz algebra does
not necessarily give a Lie algebra, as the Jacobi identity is not necessarily 
satisfied.               

\begin{lem}   \label{derivations_Leibniz}
For any Leibniz algebra ${\mathfrak h}$, the space of derivations ${\rm der}({\mathfrak h})$ 
together with the bracket of derivations
$$[D,D\,']\,:=\,D\circ D\,'-D\,'\circ D,$$
forms a Lie algebra. 
\end{lem}

\pr The bracket satisfies the Jacobi identity because of the associativity of the 
composition of endomorphisms of ${\mathfrak h}$. The bracket of derivations is once again 
a derivation by the following computation for all $X,Y\in{\mathfrak h}$:
\begin{eqnarray*}
[D,D\,']([X,Y])&=&(D\circ D\,')([X,Y])-(D\,'\circ D)([X,Y]) \\
&=& [(D\circ D\,')(X),Y]+[D(X),D\,'(Y)]+[D\,'(X),D(Y)]+ \\
&+& [X,(D\circ D\,')(Y)]-[(D\,'\circ D)(X),Y]-[D(X),D\,'(Y)]\\
&-&[D\,'(X),D(Y)]- [X,(D\,'\circ D)(Y)]\\
&=& [(D\circ D\,'-D\,'\circ D)(X),Y]+[X,(D\circ D\,'-D\,'\circ D)(Y)] \\
&=& [[D,D\,'](X),Y]+ [X,[D,D\,'](Y)]\,\,\,\,\,\,\,\,\,\,\,\Box
\end{eqnarray*}

By the Leibniz identity, for all $X\in{\mathfrak h}$, the endomorphism ${\rm ad}_X$ is a derivation,
called the {\it inner derivation} associated to $X$. 

\begin{lem}
The subspace ${\rm inn}({\mathfrak h})$ of inner derivations of a Leibniz algebra ${\mathfrak h}$
forms an ideal in the Lie algebra ${\rm der}({\mathfrak h})$ of all derivations.
\end{lem}

\pr We have for all $X,Y\in{\mathfrak h}$:
\begin{eqnarray*}
[D,{\rm ad}_X](Y)&=& D([X,Y])-[X,D(Y)] \\
&=& [D(X),Y]+[X,D(Y)]-[X,D(Y)] \\
&=& [D(X),Y] \,=\, {\rm ad}_{D(X)}(Y).\,\,\,\,\,\,\,\,\,\,\,\,\,\,\,\,\Box
\end{eqnarray*}

Observe that the subspace ${\rm inn}({\mathfrak h})$ is also the image of the map
${\rm ad}:{\mathfrak h}\to{\rm der}({\mathfrak h})$.   
  
\begin{defi}
For any Leibniz algebra ${\mathfrak h}$, the quotient Lie algebra of ${\rm der}({\mathfrak h})$
by the ideal of inner derivations ${\rm inn}({\mathfrak h})$ is called the Lie algebra 
${\rm out}({\mathfrak h})$ of outer derivations of ${\mathfrak h}$.
\end{defi}

\begin{defi}
The left center of a Leibniz algebra ${\mathfrak h}$ is the subspace
$$Z_L({\mathfrak h})\,:=\,\{\,X\in {\mathfrak h}\,|\,[X,Y]=0\,\,\,\forall\,\,Y\in{\mathfrak h}\,\}.$$
\end{defi} 

\begin{lem}
The left center $Z_L({\mathfrak h})$ of a Leibniz algebra ${\mathfrak h}$ is an abelian (left) ideal.
\end{lem}

\pr By the Leibniz identity, we have for all $X,Z\in{\mathfrak h}$ and $Y\in Z_L({\mathfrak h})$
$[[X,Y],Z]\,=\,[X,[Y,Z]]-[Y,[X,Z]]\,=\,0$. \fin

We summarize the preceeding discussion in the following

\begin{prop}
For any Leibniz algebra ${\mathfrak h}$, there is a 4-term exact 
sequence of Leibniz algebras:
$$0\to Z_L({\mathfrak h})\to {\mathfrak h}\stackrel{{\ad}}{\to} {\rm der}({\mathfrak h}) \to
{\rm out}({\mathfrak h})\to 0.$$
The only Leibniz algebra in this sequence which is not necessarily a Lie algebra is ${\mathfrak h}$.
\end{prop}

We can shorten the 4-term sequence to the following short exact sequence:
\begin{equation}   \label{the_abelian_extension}
0\to Z_L({\mathfrak h})\to {\mathfrak h}\to {\rm ad}({\mathfrak h})\to 0.
\end{equation}
In this way, we can associate to each Leibniz algebra ${\mathfrak h}$ an abelian extension such that
the quotient algebra (here ${\rm ad}({\mathfrak h})$) is a Lie algebra. There are of course other
choices which satisfy this requirement, but we will always use this one. 

Now let us come to automorphisms and the exponential map. 

\begin{defi}
A linear map $\alpha:{\mathfrak h}\to{\mathfrak h}$ on a Leibniz algebra ${\mathfrak h}$
is called an endomorphism 
in case for all $X,Y\in{\mathfrak h}$:
$$\alpha([X,Y])\,=\,[\alpha(X),\alpha(Y)].$$ 
Such a map $\alpha$ is called 
an automorphism if in addition it is bijective. 
\end{defi}

We will specialize from now on to $k=\R$ (although this is not completely necessary).

\begin{lem}   \label{exponential_of_derivation}
Let ${\mathfrak h}$ be a Leibniz algebra. Suppose that for a derivation $D\in{\rm der}({\mathfrak h})$
the formula 
$$\exp(D)\,:=\,\sum_{k=0}^{\infty}\frac{1}{k!}D^k$$
defines an endomorphism of ${\mathfrak h}$. Then $\exp(D)$ is an automorphism of ${\mathfrak h}$.
\end{lem}

\pr The formula $D([X,Y])\,=\,[D(X),Y]+[X,D(Y)]$ for all $X,Y\in{\mathfrak h}$ leads by induction to 
$$\frac{D^n}{n!}([X,Y])\,=\,\sum_{j=0}^{\infty}\left[\frac{D^j(X)}{j!},
\frac{D^{n-j}(Y)}{(n-j)!}\right].$$
From here, we obtain
\begin{eqnarray*}
[(\exp D)(X),(\exp D)(Y)]&=&\left[\sum_{p=0}^{\infty}\frac{D^p(X)}{p!},
\sum_{q=0}^{\infty}\frac{D^{q}(Y)}{q!}\right] \\
&=& \sum_{n=0}^{\infty}\sum_{j=0}^{n}\left[\frac{D^j(X)}{j!},
\frac{D^{n-j}(Y)}{(n-j)!}\right]  \\
&=& \sum_{n=0}^{\infty}\frac{D^n}{n!}([X,Y]) \\
&=& (\exp D)([X,Y]).\,\,\,\,\,\,\,\,\,\,\,\,\,\,\,\Box
\end{eqnarray*}

The condition of the preceeding lemma is fulfilled for example in case the derivation is locally
nilpotent, or in case the Leibniz algebra ${\mathfrak h}$ is finite-dimensional, as in this case 
the exponential series of an endomorphism converges (in fact, the exponential series is a polynomial
in this case !). Usually, we will consider finite-dimensional
Leibniz algebras and thus be in the second case. 

\begin{lem}  \label{conjugation_Lemma}
Let ${\mathfrak h}$ be a finite-dimensional Leibniz algebra, $\alpha\in{\rm Aut}({\mathfrak h})$ 
be an automorphism, and $X\in{\mathfrak h}$. Then the following formula holds:
$$\alpha\circ\exp({\rm ad}_X)\circ\alpha^{-1}\,=\,\exp({\rm ad}_{\alpha(X)}).$$
\end{lem}

\pr Let us first prove the infinitesimal formula:
$$\alpha\circ{\rm ad}_X\circ\alpha^{-1}\,=\,{\rm ad}_{\alpha(X)}.$$
This follows directly from the fact that $\alpha$ is an automorphism by applying the formula
to an element $Y$ of ${\mathfrak h}$:
$$\alpha\circ{\rm ad}_X\circ\alpha^{-1}(Y)\,=\,\alpha([X,\alpha^{-1}(Y)]\,=\,[\alpha(X),Y]\,=\,
{\rm ad}_{\alpha(X)}(Y).$$
Now notice that composition with $\alpha$ is continuous, i.e. $\alpha\circ(\lim_{N\to\infty}\phi_N)= 
\lim_{N\to\infty}(\alpha\circ\phi_N)$, by finite-dimensionality of ${\mathfrak h}$. We have
\begin{eqnarray*}
\alpha\circ\left(\sum_{k=0}^N\frac{({\rm ad}_X)^k}{k!}\right)\circ\alpha^{-1}&=&\sum_{k=0}^N
\frac{\alpha\circ({\rm ad}_X)^k\circ\alpha^{-1}}{k!} \\
&=&\sum_{k=0}^N\frac{(\alpha\circ{\rm ad}_X\circ\alpha^{-1})^k}{k!} \\ 
&=&\sum_{k=0}^N\frac{{\rm ad}_{\alpha(X)}}{k!} 
\end{eqnarray*}
The identity follows now from passage to the limit $\lim_{N\to\infty}$ using the continuity of 
the composition with $\alpha$.\fin 

Recall the usual naturality properties of the exponential map of a Lie algebra ${\mathfrak g}$,
see for example \cite{Var} p. 104, formula (2.13.7) and Theorem 2.13.2:

\begin{prop}    \label{naturality_properties}
The exponential map $\exp:{\mathfrak g}\to G$ of a (finite-dimensional) Lie algebra ${\mathfrak g}$ 
into a Lie group $G$ with Lie algebra ${\mathfrak g}$ has the 
following naturality properties for all $X,Y\in{\mathfrak g}$ and all $g\in G$:
\begin{enumerate}
\item[(a)] ${\rm conj}_{g}(\exp(Y))\,=\,\exp({\rm Ad}_g(Y))$,
\item[(b)] $\exp({\rm ad}_X)(Y)\,=\,{\rm Ad}_{\exp(X)}(Y)$. 
\end{enumerate}
\end{prop}

(Observe that we use the usual imprecision concerning the exponential function $\exp:{\mathfrak g}
\to G$ and the exponential function $\exp:{\mathfrak g}{\mathfrak l}({\mathfrak g})\to
{\rm Gl}({\mathfrak g})$.)  

\subsection{Leibniz algebras as subalgebras of a hemi-semi-direct product}

\begin{lem}
Let ${\mathfrak g}$ be a Lie algebra and $V$ be a ${\mathfrak g}$-module. The direct
sum $V\oplus{\mathfrak g}$ together with the bracket
$$[(v,X),(v',X')]\,=\,(X(v'),[X,X'])$$
becomes a Leibniz algebra, called the hemi-semi-direct product $V\times_{\rm hs}{\mathfrak g}$
of $V$ and ${\mathfrak g}$.
\end{lem}

It is readily verified that this bracket gives a Leibniz bracket which is Lie only if the 
${\mathfrak g}$-module is trivial. This structure came up in the search of the integration
of Courant algebroids. Indeed, when the problem is formulated algebraically, there is only ``one 
half'' of the semi-direct product Lie bracket
$$[(v,X),(v',X')]\,=\,(X(v')-X'(v),[X,X'])$$
playing a role. Thus the term hemi-semi-direct product. 

Kinyon and Weinstein showed in \cite{KinWei} that every Leibniz algebra may be embedded into
a hemi-semi-direct product Leibniz algebra.
 
\vspace{.5cm}
\noindent{\bf Example}: Let a Leibniz algebra ${\mathfrak h}$ be given. Our most important example 
of a hemi-semi-direct product is to choose ${\mathfrak g}{\mathfrak l}({\mathfrak h})$ as the Lie
algebra and ${\mathfrak h}$ as the module in the above construction. Kinyon and Weinstein noticed
that every Leibniz algebra may be embedded in this type of hemi-semi-direct product
${\mathfrak h}\times_{\rm hs}{\mathfrak g}{\mathfrak l}({\mathfrak h})$. The embedding 
map is simply $X\mapsto (X,{\rm ad}_X)$. In other words, the given Leibniz algebra ${\mathfrak h}$
is seen as a subalgebra of the hemi-semi-direct product ${\mathfrak h}\times_{\rm hs}{\mathfrak g}
{\mathfrak l}({\mathfrak h})$ by regarding it as the graph of the adjoint representation
$${\rm ad}:{\mathfrak h}\to{\mathfrak g}{\mathfrak l}({\mathfrak h}),\,\,\,\,X\mapsto{\rm ad}_X,$$
where for each $Y\in{\mathfrak h}$, ${\rm ad}_X(Y):=[X,Y]$.
\vspace{.5cm} 

One can change this example somehow by considering the Lie algebra of derivations ${\rm der}(
{\mathfrak h})$ instead of the Lie algebra ${\mathfrak g}{\mathfrak l}({\mathfrak h})$.
Notice that the derivations ${\rm der}({\mathfrak h})$ of a Leibniz algebra ${\mathfrak h}$ 
form indeed a Lie algebra by Lemma \ref{derivations_Leibniz}. We 
summarize this discussion in the following proposition, due to Kinyon-Weinstein {\it loc. cit.}: 

\begin{prop} \label{Kinyon_Weinstein}
Every Leibniz algebra ${\mathfrak h}$ is embedded as a subalgebra of the hemi-semi-direct product 
${\mathfrak h}\times_{\rm hs}{\rm der}({\mathfrak h})$.
\end{prop}

\subsection{On the BCH-formula}

The Ba\-ker\--Camp\-bell\--Haus\-dorff formula (BCH-formula) 
\begin{eqnarray*}
X*Y &:=&\log\left( \exp(X)\exp(Y)\right)\,=:  \\
=:{\rm BCH}(X,Y)&:=& X+Y+\frac{1}{2}[X,Y]+\frac{1}{12}\left([X,[X,Y]]+
[Y,[Y,X]]\right)+\ldots\,,
\end{eqnarray*}
for two elements $X,Y$ in a Lie algebra ${\mathfrak g}$ defines a {\it local} Lie group structure
in a neighbourhood of $0$ in ${\mathfrak g}$. It is in general not a global group structure, because 
the bracket expression need not converge. 

\begin{defi}   \label{BCH_neighborhood}
Let us call BCH-neighborhood a $0$-neighborhood 
$U$ in a Lie algebra
${\mathfrak g}$ with the following two properties:
\begin{enumerate}
\item The BCH series converges on $U\times U$ and defines thus a (local) group product
$*:U\times U\to{\mathfrak g}$ by
$$X*Y\,=\,\log\left( \exp(X)\exp(Y)\right).$$ 
\item There exists a $0$-neighborhood $V\subset U$ where $\exp:V\to\exp(V)$ is a diffeomorphism (onto the
open set $\exp(V)$).
\end{enumerate}
\end{defi}

It is easily shown (using the derivative of the exponential map)
that such a BCH-neighborhood exists for every (finite dimensional) Lie algebra ${\mathfrak g}$.

Now we pass to the conjugation with respect to the BCH-product. 
Note that due to associativity of the BCH-formula, we have
\begin{eqnarray*}
{\rm conj}_*(X,Y)&:=&\log\left( \exp(X)\exp(Y)\exp(-X)\right)\,=\,   \\
&=&{\rm BCH}({\rm BCH}(X,Y),-X)\,=\, \\
&=&{\rm BCH}(X,{\rm BCH}(Y,-X)).
\end{eqnarray*}
It is not so well-known that the conjugation associated to the BCH-multiplication has a much simpler
formula which converges always: 

\begin{lem}  \label{formula_BCH_conjugation}
The explicit formula BCH-conjugation ${\rm conj}_*$ for a Lie algebra ${\mathfrak g}$ is:
\begin{eqnarray*}
{\rm conj}_*(X,Y)&=&\exp({\rm ad}_X)(Y) \\
&=& Y + [X,Y] + \frac{1}{2}[X,[X,Y]] + \frac{1}{6}[X,[X,[X,Y]]]+\ldots
\end{eqnarray*}
\end{lem}

\pr By the naturality properties in Proposition \ref{naturality_properties}, we compute for all
$X,Y\in{\mathfrak g}$:
\begin{eqnarray*}
\exp(X)\exp(Y)\exp(-X)&=&{\rm conj}_{\exp(X)}(\exp(Y))\\
&=&\exp({\rm Ad}_{\exp(X)}(Y)) \\
&=&\exp(\exp({\rm ad}_X)(Y))
\end{eqnarray*}
The formula follows now from taking the formal logarithm.\fin 

This conjugation operation is thus a perfectly {\it global} operation, but which is only {\it locally}
the conjugation with respect to a group product.

Note that this operation  
$$(X,Y)\mapsto \exp({\rm ad}_X)(Y)$$
makes also sense for elements $X,Y$ in any finite dimensional Leibniz algebra ${\mathfrak h}$. The
exponential $\exp({\rm ad}_X)$ is the (inner) automorphism (see Lemma \ref{exponential_of_derivation})
with respect to the (inner) derivation ${\rm ad}_X$
which is associated to each element $X\in{\mathfrak h}$.

\subsection{Lie racks}

Recall the notion of a rack: It comes from axiomatizing the notion of conjugation in a group
and plays its role in the present context as the structure integrating Leibniz algebras. 

\begin{defi}
Let $X$ be a set together with a binary operation denoted $(x,y)\mapsto x\rhd y$
such that for all $x\in X$, the map $y\mapsto x\rhd y$ is bijective and
for all $x,y,z\in X$,
$$x\rhd(y\rhd z)\,=\,(x\rhd y)\rhd(x\rhd z).$$
Then we call $X$ (or more precisely $(X,\rhd)$) a (left) rack. In case the map $y\mapsto x\rhd y$
is not necessarily bijective for all $x\in X$, $X$ is called a (left) shelf. 
\end{defi}  

As already mentioned, an example of a rack is the conjugation in a group $G$. The 
rack operation is in this case given by $(g,h)\mapsto ghg^{-1}$. Finite racks have served to 
define knot, link and tangle invariants, see for example \cite{FenRou}. 
There is also the notion of a right rack. This is by 
definition a set $X$ together with a binary operation $(x,y)\mapsto x\lhd y$ such that all maps 
$x\mapsto x\lhd y$ are bijective and 
$$(x\lhd y)\lhd z\,=\,(x\lhd z)\lhd(y\lhd z).$$
There are at least two ways to transform a left rack into a right rack and vice-versa. The first is to
take the {\it opposite rack} $x\lhd y:=y\rhd x$, the second is to take the {\it inverse rack}
$x\lhd y:=(y\rhd-)^{-1}(x)$.

\begin{defi}  \label{definition_rack_action}
Let $R$ be a rack and $X$ be a set. We say that $R$ acts on $X$ (on the right) 
(or that $X$ is a right $R$-set) in case
for all $r\in R$, there are bijections $(\cdot r):X\to X$ such that for all $x\in X$ and all $r,r'\in R$:
$$(x\cdot r)\cdot r'\,=\,(x\cdot r')\cdot (r\lhd r').$$ 
\end{defi}
There is also the notion of a left action where the corresponding identity reads
$$r\cdot (r'\cdot x)\,=\,(r\rhd r')\cdot (r\cdot x).$$

Clearly, the adjoint action ${\rm Ad}_r:R\to R$ defined by ${\rm Ad}_r(r'):=r\rhd r'$ in a left rack $R$
is a left action of $R$ on itself. In the same way the adjoint action of a right rack on itself is a right
action. 

\begin{lem}   \label{coadjoint_action}
Let $R$ be a left rack such that the underlying set is a finite dimensional vector 
space with linear dual $R^*$. 
Then there exists
a coadjoint action ${\rm Ad}^*:R\times R^*\to R^*$ defined for all $r,r'\in R$ and all $f\in R^*$ by
$$({\rm Ad}^*_r(f))(r')\,:=\,f((r\rhd-)^{-1}(r')).$$
The coadjoint action is a left action.
\end{lem}

\pr For the proof, write simply $r\cdot f$ for ${\rm Ad}^*_r(f)$. Then
\begin{eqnarray*}
(r\cdot(r'\cdot f))(r\rhd(r'\rhd r''))&=&(r'\cdot f)((r\rhd-)^{-1}(r\rhd(r'\rhd r''))) \\
&=&(r'\cdot f)(r'\rhd r'') \\
&=& f((r'\rhd-)^{-1}(r'\rhd r'')) \\
&=& f(r'').
\end{eqnarray*} 

We also have

\begin{eqnarray*}
((r\rhd r')\cdot(r\cdot f))(r\rhd(r'\rhd r''))&=&((r\rhd r')\cdot(r\cdot f))((r\rhd r')\rhd(r\rhd r'')) \\
&=&(r\cdot f)(((r\rhd r')\rhd-)^{-1}((r\rhd r')\rhd(r\rhd r''))) \\
&=&(r\cdot f)(r\rhd r'') \\
&=&f((r\rhd-)^{-1}(r\rhd r'')) \\
&=&f(r'').
\end{eqnarray*}
This shows that 
$$r\cdot(r'\cdot f)\,=\,(r\rhd r')\cdot(r\cdot f),$$
thus the coadjoint action is a left action.
\fin 

\begin{rem}
Curiously, this does not seem to work with the opposite rack structure
replacing the inverse rack structure.
\end{rem} 

In the following, we will need pointed local Lie racks.

\begin{defi}
A pointed rack $(X,\rhd,1)$ is a set $X$ with a binary operation $\rhd$ and an element $1\in X$ 
such that the following axioms are satisfied:
\begin{enumerate}
\item $x\rhd(y\rhd z)\,=\,(x\rhd y)\rhd(x\rhd z)$ for all $x,y,z\in X$,
\item For each $a,b\in X$, there exists a unique $x\in X$ such that $a\rhd x\,=\,b$,
\item $1\rhd x\,=\,x$ and $x\rhd 1\,=\,1$ for all $x\in X$.
\end{enumerate}
\end{defi} 

Once again, the conjugation rack of a group is an example of a pointed rack. 

\begin{defi}
\begin{enumerate}
\item A Lie rack $X$ is a manifold and a pointed smooth rack, i.e. the structure maps are 
smooth.  
\item A local Lie rack is a manifold $X$ with an open subset $\Omega\subset X\times X$
where a Lie rack product $\rhd$ is defined such that 
\begin{enumerate}
\item If $(x,y),(x,z),(y,z),(x,y\rhd z),(x\rhd y,x\rhd z)\in\Omega$, then 
$x\rhd(y\rhd z)\,=\,(x\rhd y)\rhd(x\rhd z)$.
\item If $(x,y),(x,z)\in\Omega$ and $x\rhd y=x\rhd z$, then $y=z$.
\item For all $x\in X$, $(1,x),(x,1)\in\Omega$ and as usual $1\rhd x=x$ and
$x\rhd 1=1$.     
\end{enumerate}
\end{enumerate}  
\end{defi}

Examples of Lie racks include obviously the conjugation racks associated to Lie groups. Another
example which will play an important role in the sequel is the following:\\

\noindent{\bf Example:} Let $G$ be a Lie group and $V$ be a $G$-module. On $X:=V\times G$, we define 
a binary operation $\rhd$ by
$$(v,g)\rhd(v',g')\,=\,(g(v'),gg'g^{-1})$$
for all $v,v'\in V$ and all $g,g'\in G$. $X$ is a Lie rack with unit $1:=(0,1)$ which is called 
a {\it linear Lie rack}. This is the "group-analog" of the hemi-semi-direct product of a Lie algebra
with its representation, and we denote it by $V\times_{\rm hs}G$.   

Let us define more generally this hemi-semi-direct product of racks:

\begin{defi}
Let $R$ be a rack and $A$ be a rack module in the sense of Definition \ref{definition_rack_action}. 
The hemi-semi-direct product
$A\times_{\rm hs}R$ of $R$ with $A$ is the following rack structure on the direct product set
$A\times R$:
$$(a,r)\rhd(a',r')\,:=\,(r(a'),r\rhd r').$$
\end{defi}

\noindent One verifies easily that this gives indeed a rack structure.

Now let us come to digroups:

\begin{defi}
A digroup $(H,\vdash,\dashv)$ is a set $H$ together with two binary operations $\vdash$ and $\dashv$
satisfying the following axioms. For all $x,y,z\in H$,
\begin{enumerate}
\item $(H,\vdash)$ and $(H,\dashv)$ are semigroups,
\item $x\vdash(y\dashv z)\,=\,(x \vdash y)\dashv z$,
\item $x\dashv(y\vdash z)\,=\,x \dashv (y\dashv z)$,
\item $(x\dashv y)\vdash z\,=\,(x \vdash y)\vdash z$,
\item there exists $1\in H$ such that $1\vdash x\,=\,x\dashv 1\,=\,x$ for all $x\in H$,
\item for all $x\in H$, there exists $x^{-1}\in H$ such that $x\vdash x^{-1}\,=\,x^{-1}\dashv x\,=\,1$.
\end{enumerate}
\end{defi} 

An element $e\in H$ in a digroup $H$ is called a {\it bar unit} in case 
$e\vdash x\,=\,x\dashv e\,=\,x$ for all $x\in H$. Bar units exist in a digroup, but are not 
necessarily unique. A digroup is a group if and only if $\vdash\,\,=\,\,\dashv$ and $1$ is the 
unique bar unit.

There is a digroup which resembles very much the linear Lie rack:

\begin{rema}   \label{remark_linear_digroup} 
Let $G$ be a Lie group and $M$ be a $G$-module. Define on $H:=M\times G$ the structure of a digroup by
$$(u,g)\vdash(v,h)\,:=\,(g(v),gh)$$
and
$$(u,g)\dashv(v,h)\,:=\,(u,gh)$$
for all $u,v\in M$ and all $g,h\in G$. Then $M\times G$ is a Lie digroup with distinguished bar unit 
$(e,1)$. The inverse of an element $(u,g)$ is $(e,h^{-1})$. This Lie digroup is called the 
linear Lie digroup associated to $G$ and $M$.  
\end{rema} 

Digroups give rise to racks in the following way:

\begin{prop}    \label{digroup_to_rack}
Let $(H,\vdash,\dashv)$ be a digroup and put
\begin{equation}   \label{digroup_rack}
x\rhd y\,:=\,x\vdash y\dashv x^{-1}
\end{equation}
for all $x,y\in H$. Then $(H,\rhd)$ is a rack, pointed in $1$. Moreover, in case $(H,\vdash,\dashv)$ 
is a Lie digroup (i.e. all structures are smooth), $(H,\rhd)$ is a Lie rack.
\end{prop}
 
In the case of the example in Remark \ref{remark_linear_digroup}, the obtained Lie rack 
is the above described linear Lie rack $M\times_{\rm hs}G$. 
In this sense every linear Lie rack ``comes from'' a linear Lie digroup. 



\begin{rema}
There are several ways to construct a group out of a rack $X$. The {\it associated group}
$As(X)$ is the quotient of the free group on $X$ by the normal subgroup generated by
the set $\{(xy^{-1}x^{-1})(x\rhd y)\,:\,x,y\in X\}$. For pointed racks, one modifies 
this definition such that $1$ becomes the unit of $As(X)$. 
\end{rema}

\subsection{From split Leibniz algebras to Lie racks}

In this subsection, we summarize Kinyon's approach \cite{Kin}
to the integration of (split) Leibniz algebras by Lie racks.

Kinyon shows in \cite{Kin} the following theorem which is at the heart of all our attempts to
integrate Leibniz algebras. 

\begin{theo} \label{Kinyon1}
Let $(X,\rhd,1)$ be a Lie rack, and let ${\mathfrak h}:=T_1X$. Then there exists a bilinear map
$[,]:{\mathfrak h}\times{\mathfrak h}\to{\mathfrak h}$ such that
\begin{enumerate}
\item $({\mathfrak h},[,])$ is a (left) Leibniz algebra,
\item for each $x\in X$, the tangent map $\Phi(x):=T_1\phi(x)$ of the left translation map 
$\phi(x):X\to X$, $y\mapsto x\rhd y$, is an automorphism of $({\mathfrak h},[,])$,
\item if ${\rm ad}:{\mathfrak h}\to{\mathfrak g}{\mathfrak l}({\mathfrak h})$ is defined by
$Y\mapsto{\rm ad}_X(Y):=[X,Y]$, then ${\rm ad}=T_1\Phi$.
\end{enumerate}
\end{theo}

Let us recall its proof for the sake of self-containedness:

\pr We have for all $x\in X$, $\phi(x)(1)=x\rhd 1=1$, thus $\Phi(x):=T_1\phi$ is an endomorphism of 
${\mathfrak h}:=T_1X$. As each $\phi(x)$ is invertible, we have $\Phi(x)\in {\rm Gl}({\mathfrak h})$. Now
the map $\Phi:X\to {\rm Gl}({\mathfrak h})$ satisfies $\Phi(1)=\id$, thus we may differentiate again 
in order to obtain ${\rm ad}:T_1X\to {\mathfrak g}{\mathfrak l}({\mathfrak h})$. Now we set
$$[X,Y]\,:=\,{\rm ad}_X(Y)$$
for all $X,Y\in {\mathfrak h}=T_1X$. In terms of the left translations $\phi(x)$, the rack identity
can be expressed by the equation
$$\phi(x)(\phi(y)(z))\,=\,\phi(\phi(x)(y))(\phi(x)(z)).$$
We differentiate this equation at $1\in X$ first with respect $z$, then with respect to $y$ to obtain
$$\Phi(x)\left([Y,Z]\right)\,=\,[\Phi(x)(Y),\Phi(x)(Z)]$$
for all $x\in X$ and all $Y,Z\in {\mathfrak h}$. This expresses the fact that for each $x\in X$,
$\Phi(x)\in{\rm Aut}(T_1X,[,])$. Finally, we differentiate this last equation at $1$ with respect to $x$ 
to obtain
$$[X,[Y,Z]]\,=\,[[X,Y],Z]+[Y,[X,Z]]$$
for all $X,Y,Z\in {\mathfrak h}$, This shows that ${\mathfrak h}$ is a left Leibniz algebra.\fin

\noindent{\bf Example:} In the special case of a linear Lie rack, we obtain the hemi-semi-direct product
Leibniz algebra ${\mathfrak h}=V\times_{\rm hs}{\mathfrak g}$, where ${\mathfrak g}$ is the 
Lie algebra of the Lie group $G$, endowed with the bracket:
$$[(v,X),(v',X')]\,=\,(X(v'),[X,X']).$$
The $G$-module $V$ is here seen as a ${\mathfrak g}$-module in the usual way.\\

Kinyon's main result in \cite{Kin} is the integration of split Leibniz algebras (i.e. those 
isomorphic to a hemi-semi-direct product Leibniz algebra) into linear Lie racks
and thus into Lie digroups. 

\begin{theo}[Kinyon]   \label{Kinyon}
Let ${\mathfrak h}$ be a split Leibniz algebra. Then there exists a linear Lie digroup with tangent 
Leibniz algebra isomorphic to ${\mathfrak h}$. 
\end{theo}

\begin{rema} \label{remarque_Simon}
In fact, Simon Covez showed in his (unpublished) Master thesis that conversely, in case a 
Leibniz algebra integrates into a Lie digroup, it must be split over some ideal containing the
ideal of squares (more precisely, it is split over the ideal $\ker(T_1i)$ where $i$ is the 
inversion map of the digroup). 
\end{rema} 
  
\subsection{From Lie algebras to Lie racks}

Here we summarize some of the previous results in order to perform the integration of Lie algebras
into Lie racks. Later on, we will generalize these integrations to Leibniz algebras. 

Let ${\mathfrak g}$ be a Lie algebra, and let $G$ denote a Lie group integrating ${\mathfrak g}$. 
Then there are two global Lie racks:
\begin{enumerate}
\item $\rhd:G\times G\to G$ given by $g\rhd h\,=\,ghg^{-1}$,
\item $\rhd:{\mathfrak g}\times{\mathfrak g}\to{\mathfrak g}$ given by $X\rhd Y\,=\,
\exp({\rm ad}_X)(Y)$.
\end{enumerate}
Let us denote these two Lie racks by $R_G$ and $R_{\mathfrak g}$ respectively. 

\begin{prop}
There is a rack morphism $\phi:R_{\mathfrak g}\to R_G$, induced by the geometric exponential
$\phi(X):=\exp(X)$, such that $\phi$ is an isomorphism in some $0$-neighborhood. The racks $R_G$ 
and $R_{\mathfrak g}$ thus define the same local Lie rack integrating ${\mathfrak g}$.  
\end{prop}

\pr The fact that $\phi$ is a rack morphism follows from Proposition \ref{naturality_properties}. 
The fact that $\phi$ is an isomorphism in some $0$-neighborhood follows from the existence of 
BCH-neighborhoods, see Definition \ref{BCH_neighborhood}. \fin

Two remarks are in order:

\begin{rema}
The rack $R_{\mathfrak g}$ has been introduced by H. Bass (unpublished to our knowledge, but cited
in \cite{FenRou}). 
\end{rema}

\begin{rema}
For exponential Lie groups, $\phi$ is a global isomorphism. This is the case for example for 
simply connected, nilpotent Lie groups $G$.
\end{rema}

\section{Local integration using abelian extensions}

In this section, we sketch Covez' approach \cite{Cov} to the integration of Leibniz algebras. It is
modeled on the homological proof of Lie's third Theorem which we sketch first. Covez integrates
Leibniz algebras into local Lie racks by associating to each Leibniz algebra an abelian 
extension and then integrating locally the corresponding Leibniz cocycle to a rack cocycle. 

\subsection{Homological proof of Lie's third theorem}

\begin{rem}
Recall the homological proof of Lie's third theorem (cf \cite{Tuy}) using central extensions: 
write a given
finite dimensional real Lie algebra ${\mathfrak g}$ as a central extension 
$$0\to Z({\mathfrak g})\to {\mathfrak g}\to {\mathfrak g}_{\rm ad}\to 0,$$
where $Z({\mathfrak g})$ is the center of ${\mathfrak g}$ and 
${\mathfrak g}_{\rm ad}:={\mathfrak g}\,/\,Z({\mathfrak g})$ is the adjoint Lie 
algebra associated to ${\mathfrak g}$. As the center of ${\mathfrak g}$ is the kernel
of the adjoint representation, ${\mathfrak g}_{\rm ad}$ embeds into ${\mathfrak g}
{\mathfrak l}({\mathfrak g})$, the Lie algebra of endomorphisms of the vector space
${\mathfrak g}$, via the adjoint representation. As ${\mathfrak g}{\mathfrak l}
({\mathfrak g})$ is the Lie algebra of the Lie group ${\rm G}{\rm l}({\mathfrak g})$, 
this subalgebra integrates by Lie's first theorem to a connected Lie subgroup 
$G_{\rm ad}$ (in order to have simply connected groups, one might want to pass to 
the universal cover). Trivially, the vector space $Z({\mathfrak g})=:V$ integrates to itself,
seen now as a trivial module of the Lie group $G_{\rm ad}$. The above central extension
is determined by a Lie algebra $2$-cocycle which may be integrated into a locally smooth group
$2$-cocycle $\gamma$ (thanks to vanishing of the homotopy groups $\pi_1$ and $\pi_2$ of the group; we do not
review how this is done as the integration in the Leibniz case is quite different. We refer to
Neeb's paper \cite{KHN} for the Lie case.), which then 
gives rise to a central extension 
$$0\to V\to V\times_{\gamma}G_{\rm ad}\to G_{\rm ad}\to 1.$$  
This central extension is the Lie group into which the Lie algebra ${\mathfrak g}$
integrates. As a set, $V\times_{\gamma}G_{\rm ad}$ is the direct product. 
The topology and manifold structure on $V\times_{\gamma}G_{\rm ad}$ is given by
Proposition 18, Chapter III.9 (p. 226) in \cite{Bou}:

\begin{theo}  \label{Bourbaki}
Let $G$ be a group, $W\subset G$ be a subset containing the neutral element $1$ and let $W$
be endowed with a manifold structure. Assume that there exists an open neighborhood $Q\subset W$
of $1$ with $Q^{-1}=Q$ and $Q\cdot Q\subset W$ such that
\begin{enumerate}
\item the map $Q\times Q\to W$, $(g,h)\mapsto gh\in W$ is smooth,
\item the map $Q\to Q$, $g \mapsto g^{-1}$ is smooth,
\item $Q$ generates $G$ as a group.
\end{enumerate}
Then there exists a Lie group structure on $G$ such that $Q$ is open. Any other choice of $Q$
satisfying the above conditions leads to the same structure.
\end{theo}

The integrated group cocycle is globally a cocycle and thus defines a global group structure
on $G:=V\times_{\gamma}G_{\rm ad}$, therefore the theorem applies to our case to give a Lie 
group structure on $G$. As $W$, one may take an open set where the cocycle is smooth.    
\end{rem}

Starting to transpose this scheme to the framework of Leibniz algebras, by the sequence 
(\ref{the_abelian_extension}), every Leibniz
algebra is an abelien extension in the category of Leibniz algebras, i.e. the 
corresponding $2$-cocycle is a Leibniz cocycle of a Lie algebra by some module.
One can choose many ideals $I$ in a given Leibniz algebra ${\mathfrak h}$ such that
the quotient ${\mathfrak h}/I$ is a Lie algebra. Actually, every ideal $I$ which contains
the ideal (right) generated by the squares $[X,X]$ for all $X\in{\mathfrak h}$ works. 
As mentioned before, for the integration theory, we will work with the left center 
$Z_L({\mathfrak h})$ where the quotient 
${\mathfrak h}/Z_L({\mathfrak h})=:{\mathfrak h}_{\rm Lie}={\rm ad}({\mathfrak h})$ is a 
Lie algebra, and we have the abelian extension (cf the sequence (\ref{the_abelian_extension}))
of Leibniz algebras
\begin{equation} \label{abelian_extension}
0\to Z_L({\mathfrak h})\stackrel{i}{\to} {\mathfrak h} \stackrel{\pi}{\to} 
{\mathfrak h}_{\rm Lie}\to 0.
\end{equation}
As in the theory of Lie algebras, every abelian extension of Leibniz algebras is 
(uniquely up to equivalence of extensions) specified by its cohomology class which
is represented by a Leibniz $2$-cocycle $\omega:{\mathfrak h}_{\rm Lie}\times
{\mathfrak h}_{\rm Lie}\to Z_L({\mathfrak h})$. This cocycle $\omega$ is obtained exactly
as in Lie algebra cohomology, i.e. given a linear section $s:{\mathfrak h}_{\rm Lie}\to
{\mathfrak h}$, $\omega:{\mathfrak h}_{\rm Lie}\times{\mathfrak h}_{\rm Lie}\to Z_L({\mathfrak h})$
is defined for all $X,Y\in{\mathfrak h}_{\rm Lie}$ by
$$\omega(X,Y)\,=\,s([X,Y])-[s(X),s(Y)]\,\,\in\,\,\ker(\pi)=\im(i)\,\cong\,Z_L({\mathfrak h}).$$
Details about this correspondence, for example the independence (up to coboundary) of
the choice of the section, can be found in \cite{Cov}. The fact that $\omega$ is a 
Leibniz $2$-cocycle means for all $X,Y,Z\in{\mathfrak h}_{\rm Lie}$ that
$$X\cdot\omega(Y,Z)-Y\cdot\omega(X,Z)
-\omega([X,Y],Z)+\omega(X,[Y,Z])-\omega(Y,[X,Z])=0.$$
Observe that this is close to the usual Lie algebra cocycle identity, but with two 
modifications: the term $[X_i,X_j]$ takes here the place of $X_j$ and the last element
acts via the right representation. This right representation is zero in our case and we thus 
consider {\it antisymmetric} Leibniz modules. There are also {\it symmetric} Leibniz modules 
where the right module map is the negative of the left module map.

\subsection{The local Lie rack}

In \cite{Cov}, Covez uses the above mentioned ideas to integrate 
(finite dimensional real)
Leibniz algebras into local Lie racks.  
Using the group $As(X)$ for a given rack $X$, we can define the rack modules which we will need.

\begin{defi}
Let $X$ be a rack and $A$ be an abelian group equipped with a left action of the group
$As(X)$. We call $A$ an anti-symmetric homogeneous $X$-module. 
\end{defi}

In general, a rack module is a family of abelian groups and 
there are two operations, 
one from the left, one from the right. The underlying abelian group may then change 
according to an action of $X$ on the indexing set of the family. {\it Homogeneous} means
that the family consists only of one member and {\it antisymmetric} means
that the right action is trivial. It is the fact that the left center is acted on trivially 
from the right that entails that the integration will be in terms of antisymmetric rack 
modules.

Now we can integrate the (finite dimensional) Lie algebra ${\mathfrak h}_{\rm Lie}$
into a connected, simply-connected Lie group $G_0$, integrate the 
${\mathfrak h}_{\rm Lie}$-module $Z_L({\mathfrak h})$ (which is a Lie algebra module !)
into a $G_0$-module $V$ (acting on the same underlying vector space), and the last step 
to perform is the integration of the Leibniz $2$-cocycle $\omega$ into a rack $2$-cocyle
$I^2(\omega)$. This last step works only locally, i.e. the cocycle $I^2(\omega)$ is 
only defined on an open neighborhood of $(1,1)\in G_0\times G_0$. It is not like a locally
smooth group cocycle a global cocycle which is only locally smooth, but it is not even
globally a cocycle. Therefore, we cannot just transpose Theorem \ref{Bourbaki} to the rack case.

The outcome is then the 
local Lie rack $V\times_{I^2(\omega)}G_0$ defined as the abelian extension of racks
$$0\to V \to V\times_{I^2(\omega)}G_0 \to G_0\to 1.$$

\subsection{Integrating the Leibniz cocycle}

We now review the procedure for integrating the Leibniz $2$-cocycle $\omega$ into 
a rack $2$-cocycle $I^2(\omega)$. The main idea here is to integrate the two arguments
separately. Indeed, let us first integrate Leibniz $1$-cocycles into rack $1$-cocycles,
and only on the connected, simply-connected group $G$ corresponding to the Lie algebra
${\mathfrak h}_{\rm Lie}$ (and seen as conjugation rack). Choose
for all group elements $g\in G$ smooth paths $\gamma_g$ from $1$ to $g$. For a given 
Leibniz $1$-cocycle $\omega\in ZL^1({\mathfrak h}_{\rm Lie},{\mathfrak a}^s)$, define 
the equivariant $1$-form $\omega^{\rm eq}$ as the unique differential $1$-form on $G$ such 
that for all $g\in G$ and all tangent vectors $m$ at $g$ 
$$\omega^{\rm eq}(g)(m)\,=\,g\cdot(\omega(T_gL_{g^{-1}}(m))).$$
Observe that we suppose here that the Leibniz module ${\mathfrak a}$ is symmetric, i.e. the 
right module map is the negative of the left module map.       
The integration map $I^1:ZL^1({\mathfrak h}_{\rm Lie},{\mathfrak a}^s)\to ZR^1(G_0,{\mathfrak a}^s)$
is then given by 
$$I^1(\omega)(g)\,:=\,\int_{\gamma_g}\omega^{\rm eq}.$$
It is shown in {\it loc. cit.} that this map does not depend on the choice of $\gamma$, that it 
induces a map in cohomology, and that it is left inverse to the differentiation map which sends
locally smooth rack cocycles to Leibniz cocycles.

The second step is then to use the separation-of-variables-isomorphism 
$$ZL^2({\mathfrak h}_{\rm Lie},{\mathfrak a}^a)\,\cong\,
ZL^1({\mathfrak h}_{\rm Lie},\Hom({\mathfrak h}_{\rm Lie},{\mathfrak a})^s).$$
This isomorphism works for general Leibniz algebras and sends coboundaries to coboundaries.
It sends $2$-cocycles with values in anti-symmetric modules to $1$-cocycles with values in the 
symmetric module $\Hom({\mathfrak h}_{\rm Lie},{\mathfrak a})$. This is indicated by "a" and "s"
in the exponent. Composing with the integration map from the first step, we obtain a map
$$I:ZL^2({\mathfrak h}_{\rm Lie},{\mathfrak a}^a)\to ZR^1(G_0,
\Hom({\mathfrak h}_{\rm Lie},{\mathfrak a})^s).$$

The third step is to define a map from $ZR^1(G_0,\Hom({\mathfrak h}_{\rm Lie},{\mathfrak a})^s)$
to $ZR^2(U,{\mathfrak a}^a)$, where $U$ is an open $1$-neighborhood in $G_0$ such that the 
logarithm $\log$ as an inverse diffeomorphism to the exponential map $\exp$ is defined.
This time, the integration of some $\beta\in ZR^1(G_0,\Hom({\mathfrak h}_{\rm Lie},{\mathfrak a})^s)$
is similarly to the previous integration map given by
$$\int_{\gamma_{g\rhd h}}(\beta(g))^{\rm eq},$$
where obviously $g\rhd h$ for $g,h\in G_0$ means $ghg^{-1}$. 

Putting all steps together, the map $I^2:ZL^2({\mathfrak h}_{\rm Lie},{\mathfrak a}^a)\to 
ZR^2(U,{\mathfrak a}^a)$ is given by
$$I^2(\omega)(g,h)\,=\,\int_{\gamma_{g\rhd h}}(I(\omega)(g))^{\rm eq}.$$
Unfortunately, this does not work in general and we need very specific paths in the group $G_0$
to make this work. The paths take the form 
\begin{equation}   \label{path}
\gamma_g(s)\,=\,\exp(\,s\,\log(g)),
\end{equation}
and this is why we work on a $1$-neighborhood where $\log$ is defined. Covez shows that this 
map $I^2$ is well-defined (using the above exponential paths), that it sends coboundaries to 
coboundaries, and that it is a left inverse of the differentiation map. Furthermore, Covez shows
that in the case of Lie algebra cocycles the result is the image of group cocycle under the map
linking group and rack cohomology of a group. Observe that thanks to this, 
every Lie subalgebra of ${\mathfrak h}$ becomes integrated into a Lie group,
seen as a subrack of this Lie rack.

Observe that for some Lie groups (like for example simply-connected, nilpotent Lie groups), 
the exponential is a global isomorphism and thus for these groups, the integration procedure yields
a global Lie rack.    

\section{Other approaches to the integration of Leibniz algebras}

Here we present two other approaches to the integration of Leibniz algebras. They are closely related 
to work by H. Bass (unpublished), referred to in \cite{FenRou}, and M. Kinyon \cite{Kin}. 

\subsection{Bass' approach to integration}

This approach builds on a remark by H. Bass in the Lie algebra case, referred to in \cite{FenRou},
and is already contained in \cite{Kin} (end of Section $3$), but 
Kinyon believed this integration to be too arbitrary, as it does not necessarily 
yield Lie groups in the case of Lie algebras.

Let ${\mathfrak h}$ be a finite-dimensional real Leibniz algebra. 

\begin{theo}    \label{Bass_integration}
On the vector space ${\mathfrak h}$, there exists a Lie rack structure which is given by   
$$(X,Y)\mapsto \exp({\rm ad}_X)(Y)\,=:\,X\rhd Y$$
for all $X,Y\in{\mathfrak h}$. This global Lie rack structure has the following properties:
\begin{enumerate}
\item In case ${\mathfrak h}$ is a Lie algebra, the corresponding Lie rack structure is locally the 
conjugation rack structure with respect to to a Lie group structure. 
\item The Lie rack structure is globally (!) described by a BCH-formula.  
\end{enumerate}
\end{theo}

\pr Note that by Lemma \ref{exponential_of_derivation}, $X\mapsto \exp({\rm ad}_X)$ is an automorphism
of ${\mathfrak h}$. The fact that the binary operation
$$(X,Y)\mapsto X\rhd Y=\exp({\rm ad}_X)(Y)$$
is a rack product thus follows from the self-distributivity of the linear rack 
${\mathfrak h}\times_{\rm hs}{\rm Aut}({\mathfrak h})$ by projection onto the first component. 

The BCH-formula which is referrred to in the statement is contained in Lemma 
\ref{formula_BCH_conjugation}, while the local Lie group structure in the case of a Lie algebra
is given by the BCH-product.\fin  

One drawback of this Lie rack structure is that the underlying space is contractible. This 
will be different with the following approach. Another drawback is that in the case of a Lie algebra,
the space is only {\it locally} a Lie group, but not necessarily globally. We will not be able to 
overcome this drawback.

\subsection{hs-approach to integration}

The hs-approach (approach using {\bf h}emi-{\bf s}emi-direct products) can be seen as modeled on 
the proof of Lie's third Theorem using Ado's Theorem. Here we embed Leibniz algebras
as subalgebras of hemi-semi-direct products (taking the place of general linear Lie algebras),
integrate these to linear Lie racks and identify then the subrack associated to the given 
Leibniz algebra.

From the point of view of abelian extensions, a hemi-semi-direct product 
Leibniz algebra is a trivial extension, so integrates without integrating any cocycle. 

Indeed, let a hemi-semi-direct product $V\times_{\rm hs}{\mathfrak g}$ be given. The left center
$$Z_L(V\times_{\rm hs}{\mathfrak g})\,=\,\{(v,X)\in V\oplus{\mathfrak g}\,:\,\forall(v',X')
\in V\oplus{\mathfrak g}\,\,\,[(v,X),(v',X')]=0\}.$$
Recalling that $[(v,X),(v',X')]=(X(v'),[X,X'])$, we thus see that 
$$Z_L(V\times_{\rm hs}{\mathfrak g})=V\oplus\left({\rm Ann}_{\mathfrak g}(V)\cap 
Z_L({\mathfrak g})\right),$$ 
where ${\rm Ann}_{\mathfrak g}(V)=\{X\in{\mathfrak g}\,:\,\forall v\in V\,\,\,X(v)=0\}$. 

Thus when we specify to ${\mathfrak g}={\mathfrak g}{\mathfrak l}({\mathfrak h})$ for some 
Leibniz algebra ${\mathfrak h}$, we will have ${\rm Ann}_{{\mathfrak g}{\mathfrak l}({\mathfrak h})}
({\mathfrak h})=0$, and simply 
$$Z_L({\mathfrak h}\times_{\rm hs}{\mathfrak g}{\mathfrak l}({\mathfrak h}))\,\cong\,{\mathfrak h}.$$
The same conclusion holds obviously for ${\rm der}({\mathfrak h})$ instead of 
${\mathfrak g}{\mathfrak l}({\mathfrak h})$. 

Now in order to compute the cocycle, we have to choose a section $s:{\rm der}({\mathfrak h})\to
{\mathfrak h}\times_{\rm hs}{\rm der}({\mathfrak h})$. But the linear map $X\mapsto(0,X)$ is a 
section, and it is moreover a morphism of Lie and Leibniz algebras. Therefore the cocycle 
which we can compute from $s$ is zero, and by independence of the choice of the section, the 
abelian extension associated to the Leibniz algebra ${\mathfrak h}\times_{\rm hs}
{\rm der}({\mathfrak h})$ is trivial. 

It therefore integrates to a (global) Lie rack ${\mathfrak h}\times_{\rm hs}
{\rm Aut}({\mathfrak h})$. Let us summarize the above discussion in the 
following proposition:

\begin{prop}
\begin{enumerate}
\item Let ${\mathfrak g}$ be a finite dimensional Lie algebra and $V$ be a finite-dimensional
${\mathfrak g}$-module.
Then the hemi-semi-direct product $V\times_{\rm hs}{\mathfrak g}$ integrates into the (global) Lie rack
$V\times_{\rm hs}G$ where $G$ is the connected, 1-connected Lie group associated to ${\mathfrak g}$.
\item Let ${\mathfrak h}$ be a finite-dimensional Leibniz algebra. Then the hemi-semi-direct products
${\mathfrak h}\times_{\rm hs}{\mathfrak g}{\mathfrak l}({\mathfrak h})$ and 
${\mathfrak h}\times_{\rm hs}{\rm der}({\mathfrak h})$ integrate into the (global) Lie racks 
${\mathfrak h}\times_{\rm hs}{\rm Gl}({\mathfrak h})$ and  
${\mathfrak h}\times_{\rm hs}{\rm Aut}({\mathfrak h})$ respectively. 
\end{enumerate}
\end{prop}

\pr In the first setting, we need that the ${\mathfrak g}$-module $V$ integrates into a $G$-module.
This follows from the 1-connectedness, and it is here that we use finite-dimensionality.\fin         
 
\subsection{Integrating arbitrary Leibniz algebras in the hs-approach}

In this subsection, we show how to integrate a finite-dimensional Leibniz algebra 
${\mathfrak h}$ into a global hemi-semi-direct product linear Lie rack. Modulo
the steps which were performed in the previous subsection, it remains to identify the subrack
$R_{\mathfrak h}\subset{\mathfrak h}\times_{\rm hs}{\rm Aut}({\mathfrak h})$ associated to the 
Leibniz subalgebra $\{(X,{\rm ad}_X)\,:\,X\in{\mathfrak h}\}$
of the hemi-semi-direct product Leibniz algebra 
${\mathfrak h}\times_{\rm hs}{\rm der}({\mathfrak h})$. For this, we use the exponential function.

\begin{prop}
The subrack $R_{\mathfrak h}\subset{\mathfrak h}\times_{\rm hs}{\rm Aut}({\mathfrak h})$ 
is explicitely described as 
$$R_{\mathfrak h}\,=\,\{(X,\exp({\rm ad}_X))\,:\,X\in{\mathfrak h}\}.$$
It is a closed subset of the direct product of the vector space ${\mathfrak h}$ and 
the exponential image $\exp({\rm ad}({\mathfrak h}))$ of the adjoint image of ${\mathfrak h}$
in the Lie group ${\rm Gl}({\mathfrak h})$. It acquires therefore a manifold structure on some
dense open subset.
\end{prop}

\pr First we have to show is that the set  
$$R_{\mathfrak h}\,:=\,\{(X,\exp({\rm ad}_X))\,:\,X\in{\mathfrak h}\}$$ 
is a subrack of the hemi-semi-direct product rack ${\mathfrak h}\times_{\rm hs}{\rm Aut}({\mathfrak h})$.
This is clear in the first variable, and follows from the formula
$$\alpha\exp({\rm ad}_X)\alpha^{-1}\,=\,\exp({\rm ad}_{\alpha(X)})$$
for any automorphism $\alpha\in{\rm Aut}({\mathfrak h})$ in the second variable, see Lemma
\ref{conjugation_Lemma}.   

The fact that the exponential image contains a dense open set where it has a manifold structure
follows from the fact that the vanishing of the derivative of the exponential functions 
defines a strictly lower dimensional submanifold.  \fin

We summarize the content of these two subsections in the following theorem:

\begin{theo}  \label{theorem_rack_R_h}
For every (real) Leibniz algebra ${\mathfrak h}$, there exists a rack $R_{\mathfrak h}$
which carries the structure of a Lie rack on some dense open set whose tangent Leibniz algebra
is ${\mathfrak h}$. This Lie rack structure has the following properties:
\begin{enumerate}
\item In case ${\mathfrak h}$ is a Lie algebra, the corresponding Lie rack structure is locally the 
conjugation rack structure with respect to to a Lie group structure. 
\item The Lie rack structure is globally (!) described by a BCH-formula.   
\end{enumerate}
\end{theo}

\pr The first property follows from the fact that locally, $R_{\mathfrak h}$ is isomorphic to the 
rack ${\mathfrak h}$ described in Theorem \ref{Bass_integration}. This can be seen 
by explicitely by constructing 
a rack morphism $\phi:{\mathfrak h}\to R_{\mathfrak h}$ where ${\mathfrak h}$ carries the Bass rack
structure
$$X\rhd Y\,:=\,\exp({\rm ad}_X)(Y)$$
for all $X,Y\in{\mathfrak h}$. The map $\phi$ is then defined by
$$\phi(X)\,:=\,(X,\exp({\rm ad}_X)).$$
$\phi$ is a rack morphism because 
$$\exp({\rm ad}_{\exp({\rm ad}_X)(Y)})\,=\,\exp({\rm ad}_X)\exp({\rm ad}_Y)\exp(-{\rm ad}_X),$$
which follows easily from Lemma \ref{conjugation_Lemma}. $\phi$ is thus an isomorphism. 

The explicit BCH-description of the rack product is
$$(X,\exp({\rm ad}_X))\rhd(Y,\exp({\rm ad}_Y))\,=\,(\sum_{k=0}^{\infty}\frac{1}{k!}({\rm ad}_X)^k(Y),
\exp({\rm ad}_X)\exp({\rm ad}_Y)\exp(-{\rm ad}_X)).$$
This shows that the rack product is completely described in terms of the Leibniz bracket of
${\mathfrak h}$. Also without using the isomorphism $\phi$, the first property follows from Lemma 
\ref{formula_BCH_conjugation}.\fin 

\begin{rema}
The Lie racks $R_{\mathfrak h}$ and ${\mathfrak h}$ do not come in general 
from a digroup (i.e. according to Proposition \ref{digroup_to_rack}). 
It is instructive to try axiom 4 of a digroup: it does not work for ${\mathfrak h}$, but it does
works for $R_{\mathfrak h}$, because of the second component.  

On the other hand $R_{\mathfrak h}$ does not come from a digroup, because the digroup operations
make the second component different from the first. 

Observe however that Kinyon's Theorem \ref{Kinyon} states that split Leibniz algebras may be 
integrated into Lie digroups. In fact by Remark \ref{remarque_Simon}, in case a Leibniz algebra 
integrates into a Lie digroup, it is necessarily split.    
\end{rema}

\subsection{Some properties of the global Lie rack}

In the previous section, we described the rack product of the Lie rack $R_{\mathfrak h}$ 
using only the Leibniz bracket of ${\mathfrak h}$ in the spirit of the local 
description of the group product of a Lie group in terms of the Lie bracket via the BCH formula.

We now use this result to obtain a local version of Lie's Second Theorem for Lie racks of the form
$R_{\mathfrak h}$, i.e. for the diagonal Lie subracks of 
${\mathfrak h}\times_{\rm hs}{\rm Aut}({\mathfrak h})$ which were introduced earlier.
For this, note that Lie racks $R_{\mathfrak h}$ of this type have an exponential map 
$\exp:{\mathfrak h}\to R_{\mathfrak h}$
given by $X\mapsto(X,\exp({\rm ad}_X))$.

\begin{prop}  \label{Lies_second}
Let $R_1=R_{{\mathfrak h}_1}$ and $R_2=R_{{\mathfrak h}_2}$ be Lie racks of the form
$R_{\mathfrak h}$ with Leibniz algebras ${\mathfrak h}_1$ and ${\mathfrak h}_2$
respectively. Let $\alpha:{\mathfrak h}_1\to{\mathfrak h}_2$ be a morphism of Leibniz algebras.
Then there exists a unique morphism of Lie racks $\phi:R_1\to R_2$ such that
$$\phi\circ\exp(X)\,=\,\exp\circ\,\alpha(X)\,\,\,\,{\rm for}\,\,{\rm all}\,X\in {\mathfrak h}_1.$$
\end{prop}

\pr Consider the exponential map $\exp:{\mathfrak h}_i\to R_i$
given by $X\mapsto(X,\exp({\rm ad}_X))$ for $i=1,2$. Put
$$\phi:=\exp\circ\,\alpha\circ\log:R_1=\exp({\mathfrak h}_1)\to R_2.$$
It is enough to show that this map $\phi$ is a morphism of racks. 

As $\alpha$ is a morphism of Leibniz algebras, we obtain by induction
$$\alpha(X\rhd_*Y)\,=\,\alpha(X)\rhd_*\alpha(Y),$$
where the rack product $\rhd_*$ is the Bass product $X\rhd_*Y=\exp({\rm ad}_X)(Y)$. 

Now recall from the proof of Theorem
\ref{theorem_rack_R_h} that the exponential map sends the Bass product to the product in $R_{\mathfrak h}$. 
This implies directly the relation:
$$\exp(\alpha(X))\rhd\exp(\alpha(Y))\,=\,\exp(\alpha(X\rhd_* Y)).$$
Writing this relation in terms of the rack elements $(X,\exp({\rm ad}_X))$ and $(Y,\exp({\rm ad}_Y))$
using that $\log(X,\exp({\rm ad}_X))=X$ and $\log(Y,\exp({\rm ad}_Y))=Y$, one obtains
$$\phi(X,\exp({\rm ad}_X))\rhd\phi(Y,\exp({\rm ad}_Y))\,=\,\phi
\big( (X,\exp({\rm ad}_X))\rhd(Y,\exp({\rm ad}_Y))\big),$$
by observing that 
$$X\rhd_*Y\,=\,\exp({\rm ad}_X)(Y)\,=\,\log\big( (X,\exp({\rm ad}_X))\rhd(Y,\exp({\rm ad}_Y))\big).$$ 
\fin

\begin{cor}
Two Leibniz algebras ${\mathfrak h}$ and ${\mathfrak h}'$ are isomorphic if and only if their
corresponding Lie racks $R_{\mathfrak h}$ and $R_{{\mathfrak h}'}$ are isomorphic as Lie racks.
\end{cor}

\begin{rema}  \label{warning}
There is a warning in order here. The above corollary unfortunately does not necessarily apply to 
Covez' local Lie rack. We do not know whether it is locally isomorphic to our BCH-Lie rack
(but we belive strongly that it is). 

In principle, there can be different local integrations of Leibniz algebras, which 
all yield conjugation racks with respect to Lie groups in the special case of Lie algebras, but whose rack
$2$-cocycles are non cohomologuous. At the moment, we do not have an example for this instance. 
\end{rema}

\subsection{Exploring the cocycle associated to the global Lie rack}

In this subsection, we will write out explicitely the cocycle associated to the Lie rack 
$R_{\mathfrak h}$, when one choses a section of the corresponding abelian extension.

Let ${\mathfrak h}$ be a finite dimensional Leibniz algebra. 
Recall the exact sequence \ref{abelian_extension} which describes 
${\mathfrak h}$ as an abelian extension:        

$$0\to Z_L({\mathfrak h})\stackrel{i}{\to} {\mathfrak h} \stackrel{\pi}{\to} 
{\mathfrak h}_{\rm Lie}\to 0.$$

We will write $Z_L({\mathfrak h})\times_{\omega}{\mathfrak h}_{\rm Lie}$ for the Leibniz algebra
${\mathfrak h}$ when regarded as an abelian extension in this way by means of the cocycle $\omega$. 
Denote by ${\mathfrak h}\times_{{\mathfrak h}}{\rm der}({\mathfrak h})$ the diagonal subspace
of the hemi-semi-direct product
${\mathfrak h}\times_{\rm hs}{\rm der}({\mathfrak h})$, i.e. the subspace of $(X,{\rm ad}_X)$
for all $X\in{\mathfrak h}$. The diagonal subspace 
${\mathfrak h}\times_{{\mathfrak h}}{\rm der}({\mathfrak h})$ is clearly
a Leibniz subalgebra of the hemi-semi-direct product. 

\begin{prop} 
There is an isomorphism of Leibniz algebras:
$$\phi:Z_L({\mathfrak h})\times_{\omega}{\mathfrak h}_{\rm Lie}\,\cong\,{\mathfrak h}\times_{{\mathfrak h}}
{\rm der}({\mathfrak h}),$$
given by
$$(a,X)\mapsto(a+s(X),{\rm ad}_X),$$
where $s:{\mathfrak h}_{\rm Lie}\to{\mathfrak h}$ is the linear section which corresponds to the cocycle 
$\omega$. 
\end{prop}

\pr The map $\phi$ is defined by
$$(a,X)\mapsto(a+s(X),{\rm ad}_X).$$
Observe first of all that $\phi(a,X)=(a+s(X),{\rm ad}_X)=(a+s(X),{\rm ad}_{a+s(X)})$, because 
elements from $Z_L({\mathfrak h})$ act trivially on ${\mathfrak h}$. Thus $\phi$ is well-defined. 

Moreover, $\phi$
is a morphism of Leibniz algebras. Indeed, the bracket in 
the abelian extension with cocycle $\omega$ gives:
$$[(a,X),(b,Y)]\,=\,(X\cdot b+\omega(X,Y),[X,Y]),$$
which is mapped to $(X\cdot b+\omega(X,Y)+s([X,Y]),{\rm ad}_{[X,Y]})$ via $\phi$. On the other hand,
the bracket in the hemi-semi-direct product reads
$$[(a+s(X),{\rm ad}_X),(b+s(Y),{\rm ad}_Y)]\,=\,({\rm ad}_X(b+s(Y)),[{\rm ad}_X,{\rm ad}_Y]),$$
and this is equal to what we had before using ${\rm ad}_X(b)=X\cdot b$, ${\rm ad}_X(s(Y))=[s(X),s(Y)]$
(because the difference $X-s(X)$ is left central), and $\omega(X,Y)+s([X,Y])=[s(X),s(Y)]$ by definition.
But it is clear that the morphism $\phi$ is an isomorphism. 
\fin

Now we want to present the Lie rack $R_{\mathfrak h}$ in the same spirit as an abelian extension.
For details about abelian extension of racks, see e.g. \cite{Cov}.  

Every Leibniz algebra ${\mathfrak h}$ gives rise to a Lie rack $R_{\mathfrak h}$, and furthermore
to an abelian extension of racks:
$$0\to Z_L({\mathfrak h})\stackrel{I}{\to} R_{\mathfrak h}\stackrel{P}{\to}
{\mathfrak h}_{\rm Lie}\times_{{\mathfrak h}_{\rm Lie}} \exp({\rm ad}_{{\mathfrak h}_{\rm Lie}})\to 1.$$
Here the Leibniz algebra 
${\mathfrak h}_{\rm Lie}\times_{{\mathfrak h}_{\rm Lie}} \exp({\rm ad}_{{\mathfrak h}_{\rm Lie}})$
is regarded as a Lie rack by means of the Bass rack structure, and the same holds for 
$Z_L({\mathfrak h})$, which renders it a trivial subrack of $R_{\mathfrak h}$.  
The maps $I$ and $P$ are defined by $I(a)=(a,\id)$ and 
$$P(X,\exp{\rm ad}_X)=(\pi(X),\exp{\rm ad}_{\pi(X)}).$$ 
It is easy to compute that $P$ is a morphism of racks, using that $\pi$ is a continuous linear map
(between finite dimensional vector spaces).  

Next, we need a section of $P$, i.e. a map $S:{\mathfrak h}_{\rm Lie}\times_{{\mathfrak h}_{\rm Lie}} 
\exp({\rm ad}_{{\mathfrak h}_{\rm Lie}})\to R_{\mathfrak h}$ which is right inverse to $P$. 
Using the section $s$ of the map $\pi$, $S$ can be defined as
$$S(X,\exp{\rm ad}_X)\,:=\,(s(X),\exp{\rm ad}_{s(X)}).$$
The corresponding rack $2$-cocycle $f$ is then defined for all 
$x,y\in {\mathfrak h}_{\rm Lie}\times_{{\mathfrak h}_{\rm Lie}}\exp({\rm ad}_{{\mathfrak h}_{\rm Lie}})$ by:
$$f(x,y)\,:=\,S(x)\rhd S(y)-S(x\rhd y).$$
One easily computes that this gives the following expression in our situation: 
$$f(X,Y)\,=\,\exp{\rm ad}_{s(X)}(s(Y))- s\left(\exp{\rm ad}_{X}(Y)\right),$$
where we wrote simply $X\in {\mathfrak h}_{\rm Lie}$ for 
$$(X,\exp{\rm ad}_X)\in {\mathfrak h}_{\rm Lie}\times_{{\mathfrak h}_{\rm Lie}} 
\exp({\rm ad}_{{\mathfrak h}_{\rm Lie}})$$ 
and similarly for $Y$. As usual for abelian extensions, we displayed by abuse of notation only
the $Z_L({\mathfrak h})$-component of $f(x,y)$ - the other component is trivial stemming from
the fact that $P$ is a morphism of racks.     

Using the formula of Lemma \ref{formula_BCH_conjugation}, we obtain from here the expression:
$$f(X,Y)\,=\,{\rm conj}_*(s(X),s(Y))- s\left({\rm conj}_*(X,Y)\right),$$
thus the cocycle $f$ measures the default of $s$ to be compatible with the formal conjugation map. 

We have the following explicit formula for the rack cocycle $f$ in terms of the Leibniz cocycle 
$\omega$ and the section $s$ of the abelian extension of Leibniz algebras:

\begin{lem}
\begin{eqnarray*}
f(X,Y)&=&\omega(X,Y)+\frac{1}{2}\omega(X,[X,Y])+\frac{1}{6}\omega(X,[X,[X,Y]])+\ldots \\
&+& \frac{1}{2}[s(X),\omega(X,Y)]+\frac{1}{6}[s(X),\omega(X,[X,Y])]+\ldots \\
&+& \frac{1}{6}[s(X),[s(X),\omega(X,Y)]]+\ldots \\
&+& \ldots
\end{eqnarray*}
These terms are grouped here according to the number of $s(X)$ acting upon terms in $\omega$.
\end{lem}

As already stated in Remark \ref{warning}, we do not know whether this rack $2$-cocycle is cohomologuous
to Covez' rack $2$-cocycle. On the other hand, we believe that this integration formula for cocycles 
is new, even in the special case of Lie algebra $2$-cocycles.

\subsection{Summary: Integration of Leibniz algebras}  \label{summary_section}

Thus in conclusion there are (at least) three integration methods for Leibniz algebras.
Note that in general only the local Lie racks of the last two are isomorphic.

\begin{itemize}
\item The local integration of Covez \cite{Cov} integrating the Leibniz cocycle to a local rack cocycle.
This works only locally and yields local Lie groups in the case of Lie algebras. 
Moreover, the integration 
procedure is compatible with standard maps between the group-, rack-, Leibniz and Lie cohomology spaces.
\item The integration via the conjugation with respect to the BCH formula (Bass' approach), which also locally 
yields Lie groups in the special case of Lie algebras. 
It integrates a Leibniz algebra into a rack structure on the same underlying vector space. 
\item The globalization of this local integration in terms of hemi-semi-direct products. One still has
the interpretation in terms of local groups in the case of Lie algebras and one gains globality
(in the sense that the underlying topological space may be non-contractible).    
\end{itemize}

\section{Deformation quantization of Leibniz algebras}

\subsection{Motivation}

Recall that given a finite-dimensional real Lie algebra $({\mathfrak g},[,])$, 
its dual vector space ${\mathfrak g}^*$ is a smooth manifold
which carries a Poisson bracket on its space of smooth functions, defined for all 
$f,g\in{\mathcal C}^{\infty}({\mathfrak g}^*)$ and all $\xi\in{\mathfrak g}^*$
by the {\it Kostant-Kirillov-Souriau formula}
$$\{f,g\}(\xi)\,:=\,\langle \xi,[df(\xi),dg(\xi)]\rangle.$$
Here $df(\xi)$ and $dg(\xi)$ are linear functionals on ${\mathfrak g}^*$, identified with elements of
${\mathfrak g}$. The goal of the second part of this article is to define 
deformation quantization for an analoguous bracket on the dual of a Leibniz algebra. 

Let $({\mathfrak{h}},[,])$ be a (left, real, finite-dimensional)
Leibniz algebra. Its linear dual ${\mathfrak{h}}^{*}$ is still a smooth manifold. 
The smooth functions ${\mathcal{C}}^{\infty}({\mathfrak{h}}^{*})$
on ${\mathfrak{h}}^{*}$ have a natural bracket
\[
\{\,,\,\}:{\mathcal{C}}^{\infty}({\mathfrak{h}}^{*})\times
{\mathcal{C}}^{\infty}({\mathfrak{h}}^{*})\to
{\mathcal{C}}^{\infty}({\mathfrak{h}}^{*}).
\]
Namely for all $f,g\in{\mathcal{C}}^{\infty}({\mathfrak{h}}^{*})$
and all $\xi\in{\mathfrak{h}}^{*}$ 
\begin{equation}
\{f,g\}(\xi)\,:=\,\langle\xi,[df(0),dg(\xi)]\rangle\label{Poisson_bracket}
\end{equation}
At this stage, it may seem arbitrary that in comparison to the above bracket on the 
dual of a Lie algebra, we evaluated the first variable in $0$. It is an outcome (see Theorem 
\ref{bracket_computation}) of 
our deformation quantization procedure that this is the bracket which we are deforming.  
We will not introduce a different notation for this generalized bracket. We hope it 
will be clear from the context which bracket we will be talking about. 

One readily verifies that this bilinear bracket satisfies the (right) Leibniz rule
for all $f,g,h\in{\mathcal{C}}^{\infty}({\mathfrak{h}}^{*})$:
\begin{equation} \label{*} 
\{f,gh\}\,=\,\{f,g\}h+g\{f,h\}.
\end{equation}

\begin{rem}
Observe that there is a remainder of the left Leibniz rule, too.
Vector fields are derivations on the algebra of functions. Tangent vectors are
{\it pointwise derivations}, i.e. the derivation property holds when interpreted as 
immediately followed by evaluation in a point. In this sense, the 
Leibniz rule in the first variable of $\{-,-\}$ holds when interpreted as 
immediately followed by evaluation in $0$. 
\end{rem}

On the other hand, the bracket does not satisfy anymore the
(left) Leibniz identity $\{f,\{g,h\}\}\,=\,\{\{f,g\},h\}+\{g,\{f,h\}\}$
and it is certainly not necessarily skew-symmetric. 

\begin{rem}
It is natural that the generalized bracket should satisfy much weaker conditions than
a Poisson bracket on a smooth manifold. Indeed, it is shown in \cite{GraMar} that a  
bracket on a commutative associative algebra (in characteristic zero, without zero divisors)
which satisfies the Leibniz rule in both variables and the Leibniz identity is necessarily
skew-symmetric. We thank K. Uchino for bringing this fact to our attention. 
\end{rem} 

We call the bracket in (\ref{Poisson_bracket}) \textit{Leibniz-Poisson}
bracket.
\begin{defi}
A generalized Poisson manifold is a smooth manifold $M$ whose space
of smooth functions ${\mathcal{C}}^{\infty}(M)$ is endowed with a
bilinear bracket 
\[
\{\,,\,\}:{\mathcal{C}}^{\infty}(M)\times{\mathcal{C}}^{\infty}(M)\to{\mathcal{C}}^{\infty}(M)
\]
satisfying the above property (\ref{*}). 
\end{defi}

The notion of star-product \cite{BFFLS} is closely related to the
notion of Poisson manifold.

\begin{defi}
A star-product $*$ on a Poisson manifold $(M,\{,\})$ is a formal
deformation $*_{\epsilon}$ of the commutative associative product
on ${\mathcal{C}}^{\infty}(M)$, i.e. an associative product 
\[
f*_{\epsilon}g\,=\, fg+\epsilon B_{1}(f,g)+\ldots+\epsilon^{n}B_{n}(f,g)+\ldots
\]
such that the $B_{n}(-,-)$'s are bidifferential operators for all
$n\geq1$ and that the constant function $1$ is a unit. 
\end{defi}
Namely, given a star-product $*_{\epsilon}$ on $M$, 
\[
f*_{\epsilon}g\,=\, fg+\epsilon B_{1}(f,g)+\ldots,
\]
the antisymmetrization of the first terms yields a Poisson bracket
on $M$: 
\[
\{f,g\}\,=\, B_{1}(f,g)-B_{1}(g,f).
\]
One says that the star-product $*_{\epsilon}$ \textit{quantizes}
the Poisson bracket in this case.

Conversely, M. Kontsevich \cite{Kon} showed that any Poisson bracket
can be quantized (non uniquely) into a star-product.

On the other hand, the quantization of a Lie algebra ${\mathfrak{g}}$
is known to be (roughly) the data of a $*$-algebra ${\mathcal{A}}_{\mathfrak{g}}$
for which the \textit{self-adjoint element}s
\[
U_{\mathfrak{g}}\,=\,\{a\in{\mathcal{A}}_{\mathfrak{g}}\,|\, a^{*}=a\,\}
\]
form a group isomorphic to the Lie group integrating ${\mathfrak{g}}$.

A model for this quantizing $*$-algebra is the universal enveloping
algebra $U({\mathfrak{g}})$ of the Lie algebra ${\mathfrak{g}}$;
another one is the convolution algebra $C({\mathfrak{g}})$ of continuous
functions on the integrating group. Deformation quantization, by considering
$*$-algebras quantizing the Poisson structure on the dual space ${\mathfrak{g}}^{*}$,
gives yet a third model.

Namely the \textit{Gutt $*$-algebra} $({\mathcal{C}}^{\infty}({\mathfrak{g}}^{*})[[\epsilon]],*_{{\rm Gutt}})$
(see \cite{Gut}) where $\epsilon=\frac{\hbar}{2i}$ quantizes ${\mathfrak{g}}^{*}$
in the sense of deformation quantization and also has the following
properties:
\begin{enumerate}
\item The complex conjugation is the involution of the $*$-algebra 
\[
\overline{f*_{{\rm Gutt}}g}\,=\,\overline{g}*_{{\rm Gutt}}\overline{f}.
\]

\item $U_{\mathfrak{g}}\,=\,\{E_{X}\,|\, X\in{\mathfrak{g}}\}$ where $E_{X}(\xi)=e^{\frac{i}{\hbar}\langle X,\xi\rangle}$
and 
\[
E_{X}*_{{\rm Gutt}}E_{Y}\,=\, E_{BCH(X,Y)}\quad\textrm{and}\quad\overline{E_{X}}=E_{-X}.
\]

\end{enumerate}
Thus $U_{\mathfrak{g}}$ is isomorphic rather to the formal/local
group $({\mathfrak{g}},BCH)$ integrating $({\mathfrak{g}},[\,,\,])$.

In everything that follows, one can always exchange the expansion
parameters $\epsilon$ and $\hbar$ using the formula $\epsilon=\frac{\hbar}{2i}$.

In this section, we aim at quantizing a Leibniz algebra ${\mathfrak{h}}$
using techniques similar to deformation quantization.
As we will see, what we obtain is an operation 
\[
\rhd:{\mathcal{C}}^{\infty}({\mathfrak{h}}^{*})[[\epsilon]]\times{\mathcal{C}}^{\infty}({\mathfrak{h}}^{*})[[\epsilon]]\to{\mathcal{C}}^{\infty}({\mathfrak{h}}^{*})[[\epsilon]]
\]
such that the restriction of $\rhd$ to $U_{\mathfrak{h}}=\{E_{X}\,|\, X\in{\mathfrak{h}}\}$
is a rack structure $\rhd:U_{\mathfrak{h}}\times U_{\mathfrak{h}}\to U_{\mathfrak{h}}$.

Moreover, the restriction of this operation to 
\[
\rhd:U_{\mathfrak{h}}\times{\mathcal{C}}^{\infty}({\mathfrak{h}}^{*})[[\epsilon]]\to{\mathcal{C}}^{\infty}({\mathfrak{h}}^{*})[[\epsilon]]
\]
is a rack action.
\begin{rema}
In the case of a Lie algebra $({\mathfrak{g}},[\,,\,])$, one can
obtain such a quantum rack $\rhd_{{\rm Gutt}}$ from the Gutt star-product;
namely 
\[
f\rhd_{{\rm Gutt}}g\,:=\, f*_{{\rm Gutt}}g*_{{\rm Gutt}}\overline{f},
\]
whose restriction to the exponentials is 
\[
e^{\frac{i}{\hbar}X}\rhd_{{\rm Gutt}}e^{\frac{i}{\hbar}X}\,=\, e^{\frac{i}{\hbar}X}*_{{\rm Gutt}}e^{\frac{i}{\hbar}Y}*_{{\rm Gutt}}e^{-\frac{i}{\hbar}X}\,=\, e^{\frac{i}{\hbar}{\rm conj}_{*}(X,Y)},
\]
where we have used Lemma \ref{formula_BCH_conjugation}. In the same
vein, we obtain 
\[
e^{\frac{i}{\hbar}X}\rhd_{{\rm Gutt}}g\,=\, g+\epsilon B_{X}^{1}(g)+\epsilon^{2}B_{X}^{2}(g)+\cdots,
\]
where the $B_{X}^{n}$ are certain differential operators depending
on $X\in\mathfrak{g}$. The quantum rack we will obtain in the case
of a general Leibniz algebra will not coincide with the one in the
Lie algebra case, but their restrictions on exponentials will. 
\end{rema}
We start by reinterpreting the Gutt star-product quantizing a Lie
algebra $({\mathfrak{g}},[\,,\,])$ as the quantization of the symplectic
micromorphism obtained by the cotangent lift of the group operation
$m:G\times G\to G$ on the integrating Lie group. We will then follow
a similar strategy for Leibniz algebras ${\mathfrak{h}}$ by quantizing
the corresponding micromorphism obtained by the cotangent lift of
the integrating rack structure $\rhd:R_{\mathfrak{h}}\times R_{\mathfrak{h}}\to R_{\mathfrak{h}}$.

\subsection{Gutt star-product as the quantization of a symplectic micromorphism}

Let $({\mathfrak{g}},\,[\,,\,])$ be a Lie algebra with integrating
Lie group $G$. The \textit{cotangent lift} $T^{*}m$ of the group
operation $m:G\times G\to G$ is the Lagrangian submanifold 
\[
T^{*}m\,:=\,\left\{ \left((g,T_{g}^{*}R_{h}\xi),(h,T_{h}^{*}L_{g}\xi),(gh,\xi)\right):\; g,h\in G,\,\xi\in T_{gh}^{*}G\right\} 
\]
of $\overline{T^{*}G}\times\overline{T^{*}G}\times T^{*}G$, where
$R_{h}:G\to G$ and $L_{g}:G\to G$ are the usual right and left translations
on $G$, respectively. The cotangent lift $T^{*}m$ is actually the
graph of the global symplectic groupoid
\[
\xymatrix{T^{*}G\ar@<2pt>[r]^{s}\ar@<-2pt>[r]_{t} & {\mathfrak{g}}^{*}}
\]
integrating the Poisson manifold ${\mathfrak{g}}^{*}$. We refer
the reader to \cite{Can} and \cite{Wei} for more details on the
relationships between integrated Poisson data and Lagrangian submanifolds.

Seeing $T^{*}m$ as a canonical relation from $T^{*}G\otimes T^{*}G$
($\cong T^{*}G\times T^{*}G$) to $T^{*}G$ in the symplectic category,
one wishes to associate to it a Fourier integral operator (depending
on a parameter $\hbar$) from some $L^{2}({\mathfrak{g}}^{*})\otimes L^{2}({\mathfrak{g}}^{*})\to L^{2}({\mathfrak{g}}^{*})$
whose asymptotic expansion in the limit $\hbar\to0$ would yield a
star-product, in the spirit of \cite{WeiII} and \cite{WeiIII} (see
also \cite{Wei} for a more recent exposition). 

This is in general a very hard problem analytically, and it turns
out that one is more lucky by only looking at the germ of
\[
T^{*}m\subset\overline{T^{*}G}\times\overline{T^{*}G}\times T^{*}G
\]
around the graph of the diagonal map $\triangle_{{\mathfrak{g}}^{*}}:{\mathfrak{g}}^{*}\to{\mathfrak{g}}^{*}\times{\mathfrak{g}}^{*}$,
where see see $\mathfrak{g}^{*}$ as an embedded lagrangian submanifold
in $T^{*}G$, namely the fiber over the identity element. Namely,
as shown in \cite{CDWIII}, this germ is a symplectic micromorphism,
which is readily quantizable by Fourier Integral Operators (FIO) (see
\cite{CDWIV}).

\subsubsection*{Symplectic micromorphisms}

Let us recall the definition of a symplectic micromorphism (see \cite{CDWI},
\cite{CDWII}, \cite{CDWIII}, and \cite{CDWIV} for more details)
as well as some aspect of their quantization.
\begin{defi}
A \textit{symplectic micromorphism} $([L],\phi)$ from a \textit{symplectic
microfold} $[M,A]$ (i.e. a germ of a symplectic manifold around a
Lagrangian submanifold $A\subset M$, called the \textit{core} of
the microfold) to a symplectic microfold $[N,B]$ is the data of a
Lagrangian submanifold germ $[L]$ in $\overline{M}\times N$ around
the graph ${\rm gr}(\phi)$ of a smooth map $\phi:A\to B$ such that
the intersection $L\cap(A\times B)={\rm gr}(\phi)$ is clean for a
representative $L\in[L]$.
\end{defi}
The symplectic micromorphisms are the morphisms of a category, \textit{the
microsymplectic category}. We denote them by $([L],\phi):[M,A]\to[N,B]$,
and, when the symplectic microfold is $[T^{*}A,A]$, we simply write
$T^{*}A$. 

An important example of symplectic micromorphisms comes from cotangent
lifts of smooth maps between manifolds. Namely, if $\phi:B\rightarrow A$
is a smooth map, then the conormal bundle $N^{*}\textrm{gr}\phi$
of the graph of $\phi$ is a lagrangian submanifold of $T^{*}(A\times B)$.
Using the identification (Schwartz transform) between this last cotangent
bundle and $\overline{T^{*}A}\times T^{*}B$, the conormal bundle
to the graph yields a symplectic micromorphism, which we denote by
$T^{*}\phi:T^{*}A\rightarrow T^{*}B$, by taking the germ of the resulting
lagrangian submanifold 
\[
\Big\{\Big(\big(p_{A},\,\phi(x_{B})\big),\,\big((T_{x_{B}}^{*}\phi)p_{A},x_{B}\big)\Big):\,(p_{A,}\, x_{B})\in\phi^{*}(T^{*}A)\Big\}
\]
around the graph of $\phi$, and where $(p_{A},x_{A})$ and $(p_{B},x_{B})$
are the canonical coordinates on $T^{*}A$ and $T^{*}B$ respectively. 

When the target and source symplectic microfold cores are \textit{euclidean}
(i.e. when $A=\R^{k}$ and $B=\R^{l}$ for some $k\geq 1$ and $l\geq 1$), a
symplectic micromorphism from $T^{*}A$ to $T^{*}B$ can be associated
with a family of formal Fourier Integral operators from $C^{\infty}(A)[[\hbar]]$
to $C^{\infty}(B)[[\hbar]]$ using the symplectic micromorphism generating
function (see \cite{CDWIV} for a general theory of symplectic micromorphism
quantization). 

Namely, as shown in \cite{CDWII}, when the target and source symplectic
microfold cores are euclidean any symplectic micromorphism $([L],\phi)$
from $T^{*}A$ to $T^{*}B$ can be described by a generating function
germ $[S_{L}]:\phi^{*}(T^{*}A)\rightarrow\R$ around the zero section
of the pullback bundle $\phi^{*}(T^{*}A)$ as follows: There is a
representative $ $$L\in[L]$ such that
\[
\left\{\left(\big(p_{A},\,\frac{\partial S_{L}}{\partial p_{A}}(p_{A},x_{B})\big),\,\big(\frac{\partial S_{L}}{\partial x_{B}}(p_{A},x_{B}),\, x_{B}\big)\right):\,(p_{A,}\, x_{B})\in W\right\},
\]
where $W$ is an appropriate neighborhood of the zero section in $\phi^{*}(T^{*}A)$.
This generating function $S_{L}$ is unique if one requires that it
satisfies the property $S_{L}(0,x)=0$. The geometric condition on
the cleanness of the intersection in the definition above can be expressed
in terms of the generating function as follows:
\begin{equation}
\frac{\partial S_{L}}{\partial p_{A}}(p_{A},0)=\phi(x_{B})\quad\textrm{ and }\quad\frac{\partial S_{L}}{\partial x_{B}}(0,x_{B})=0.\label{eq:gen_funct_cond}
\end{equation}
In this light, one can see $S_{L}$ as a deformation of the cotangent
lift generating function, which is the first term of $S_{L}$ in a
Taylor expansion:
\[
S_{L}(p_{A},x_{B})=\langle p_{A},\phi(x_{B})\rangle+\mathcal{O}(p_{A}^{2}).
\]

\begin{rema}
Conversely, any generating function germ $[S]:\phi^{*}(T^{*}A)\rightarrow\R$
satisfying conditions \eqref{eq:gen_funct_cond} defines uniquely
a symplectic micromorphism $([L_{S}],\phi):T^{*}A\rightarrow T^{*}B$. 
\end{rema}
Now, using the generating function $S_{L}$ of the symplectic micromorphism
$([L],\phi)$ and a function germ $a:\phi^{*}(T^{*}A)\rightarrow\R$
around the zero section, one can construct a formal operator
\begin{eqnarray*}
C^{\infty}(A)[[\hbar]] & \longrightarrow & C^{\infty}(B)[[\hbar]]\\
\psi & \mapsto & Q^{a}([L],\phi)\psi
\end{eqnarray*}
by taking the stationary phase expansion of the following oscillatory
integral:
\[
\int_{T^{*}A}\chi(p_{A},x_{A})\psi(x_{A})a(p_{A},x_{B})e^{\frac{i}{\hbar}(S_{L}(p_{A},x_{B})-p_{A}x_{A})}\frac{dx_{A}dp_{A}}{(2\pi\hbar)^{n}},
\]
where $\chi$ is a cutoff function with compact support around the
critical points of the phase $S_{L}(p_{A},x_{B})-p_{A}x_{A}$ (with respect to
the integration variables) and with value $1$ on this critical locus,
which is nothing but the points in $\{(0,\phi(x_{B})):\, x_{B}\in B\}$.
Since the critical locus is contained in the zero section, the asymptotic
expansion does not depend on the cutoff functions and, hence, is well-defined.
To simplify the notation, we will abuse it slightly, and write from
now on:
\[
(Q^{a}([L],\phi))\psi(x_{B})=\int_{\R^{k}}\hat{\psi}(p_{A})a(p_{A},x_{B})e^{\frac{i}{\hbar}S_{L}(p_{A},x_{B})}\frac{dp_{A}}{(2\pi\hbar)^{k/2}},
\]
to mean the asymptotic expansion above, and where $\hat{\psi}(p_{B})$
is the asymptotic Fourier transform of $\psi$; namely,
\[
\hat{\psi}(p_{A})=\int_{\R^{k}}\psi(x_{A})e^{-\frac{i}{\hbar}p_{A}x_{A}}\frac{dx_{A}}{(2\pi\hbar)^{k/2}}.
\]

\subsubsection*{Back to the Gutt star-product}

Let us now apply the previous section result to the quantization of
the linear Poisson structure on the dual of a Lie algebra $\mathfrak{g}$.
Consider first the integrating Lie group $G$. Taking the cotangent
lift of the group operation $m:G\times G\to G$ yields a symplectic
micromorphism 
\[
([T^{*}m],\triangle_{{\mathfrak{g}}^{*}}):[T^{*}G,{\mathfrak{g}}^{*}]\otimes[T^{*}G,{\mathfrak{g}}^{*}]\to[T^{*}G,{\mathfrak{g}}^{*}],
\]
where we take the core in the source and target symplectic microfolds
to be not the cotangent bundle zero section $G$, but rather the fiber
above the identity, i.e. the dual of the Lie algebra. Identifying
$[T^{*}G,{\mathfrak{g}}^{*}]$ with $[T^{*}{\mathfrak{g}}^{*},{\mathfrak{g}}^{*}]$
(which we will denote simply by $T^{*}\mathfrak{g}^{*}$) using the
Lagrangian embedding germ 
\[
[T^{*}{\mathfrak{g}}^{*},{\mathfrak{g}}^{*}]\to[T^{*}G,{\mathfrak{g}}^{*}],\,\,\,\,\,(X,\xi)\mapsto(\exp(X),(T_{1}^{*}L_{\exp(X)})^{-1}\xi),
\]
the Lagrangian germ $[T^{*}m]$ becomes the cotangent lift of the
local group operation $BCH:\mathfrak{g}\times\mathfrak{g}\rightarrow\mathfrak{g}$,
and $([T^{*}m],\,\Delta_{\mathfrak{g}^{*}})$ becomes a symplectic
micromorphism from $T^{*}{\mathfrak{g}}^{*}\otimes T^{*}{\mathfrak{g}}^{*}$
to $T^{*}{\mathfrak{g}}^{*}$, whose underlying Lagrangian submanifold
germ coincides with the multiplication of the local symplectic groupoid
integrating the linear Poisson structure on $\mathfrak{g}^{*}$. 

This local/formal symplectic groupoid is described in \cite{CDF},
where it is shown that $T^{*}m$ can be described in term of the following
generating function germ
\[
S(X,Y,\xi)\,=\,\Big\langle\xi,BCH(X,Y)\Big\rangle
\]
as follows:
\[
T^{*}m\,=\,\left\{ \left((X,\,\frac{\partial S}{\partial X}),\,(Y,\,\frac{\partial S}{\partial Y}),\,(\frac{\partial S}{\partial\xi},\,\xi)\right):\,(X,Y,\xi)\in W\right\} 
\]
where $W$ is an appropriate neighborhood of the zero section in $T^{*}\mathfrak{g}^{*}\oplus T^{*}\mathfrak{g}^{*}$.

Once the generating function of a symplectic micromorphism is computed,
it is easy to obtain a family of (formal) FIOs quantizing it as explained
in the previous section. In the case at hand, we obtain the following family
of formal operators
\[
Q^{a}(T^{*}m):{\mathcal{C}}^{\infty}({\mathfrak{g}}^{*})[[\epsilon]]\otimes{\mathcal{C}}^{\infty}({\mathfrak{g}}^{*})[[\epsilon]]\to{\mathcal{C}}^{\infty}({\mathfrak{g}}^{*})[[\epsilon]]
\]
of the form (in the previous section notation): 
\begin{equation}
Q^{a}(T^{*}m)(f\otimes g)(\xi)\,=\,\int_{{\mathfrak{g}}\times{\mathfrak{g}}}\widehat{f}(X)\widehat{g}(Y)a(X,Y,\xi)e^{\frac{i}{\hbar}S(X,Y,\xi)}\frac{dXdY}{(2\pi\hbar)^{n}},\label{eq:FIOs}
\end{equation}
where $a$ is the germ of a smooth function on $T^{*}\mathfrak{g}^{*}\oplus T^{*}\mathfrak{g}^{*}$
around the zero section, called the \textit{amplitude} of the FIO
$Q^{a}(T^{*}m)$, and $n$ is the dimension of ${\mathfrak{g}}$.

When $a=1$ and $S$ is the generating function of $([T^{*}m],\triangle_{{\mathfrak{g}}^{*}})$,
we have that 
\[
f*_{a}g\,=\, Q^{a}(T^{*}m)(f\otimes g)
\]
coincides with the Gutt star-product \cite{BenAmar1,BenAmar2,Gut}.
For other star-products in integral form on duals of Lie algebras
as in \eqref{eq:FIOs}, we refer the reader to the work of Ben Amar
\cite{BenAmar1,BenAmar2}.
\begin{rema}
For a general amplitude $a$, $f*_{a}g$ is not necessarily associative.
\end{rema}

\subsection{Quantizing a Leibniz algebra}

Let $({\mathfrak{h}},[,])$ be a Leibniz algebra and $(R_{\mathfrak{h}},\rhd)$
its integrating Lie rack from Section $3$. The idea is to quantize
the Lagrangian relation
\[
T^{*}\rhd:T^{*}R_{\mathfrak{h}}\times T^{*}R_{\mathfrak{h}}\to T^{*}R_{\mathfrak{h}}
\]
 as we did for the group operation in the case of a Lie algebra.

As we saw in the Lie case, it is better to consider the local model,
i.e. the integrating rack 
\[
\rhd:{\mathfrak{h}}\times{\mathfrak{h}}\to{\mathfrak{h}},\,\,\,\,\,(X,Y)\mapsto e^{{\rm ad}_{X}}(Y)=:{\rm Ad}_{X}(Y)
\]
defined on ${\mathfrak{h}}$. The first step is to take the cotangent
lift of the rack operation and compute its generating function:
\begin{prop}
The cotangent lift of $\rhd$ yields a symplectic micromorphism
\[
T^{*}\rhd:T^{*}{\mathfrak{h}}^{*}\otimes T^{*}{\mathfrak{h}}^{*}\to T^{*}{\mathfrak{h}}^{*}
\]
with generating function
\[
S_{\rhd}(X,Y,\xi)\,:=\,\langle\xi,{\rm Ad}_{X}(Y)\rangle.
\]
\end{prop}
\begin{proof}
Consider the generating function 
\begin{eqnarray*}
S_{\rhd}(X,Y,\xi) & := & \langle\xi,{\rm Ad}_{X}(Y)\rangle\\
 & = & \langle\xi,Y+[X,Y]+\frac{1}{2}[X,[X,Y]]+\ldots\rangle
\end{eqnarray*}
We will denote the variables by $(X,Y)=:P$ and $\xi$, and write
accordingly $S_{\rhd}(X,Y,\xi)=S_{\rhd}(P,\xi)$.

As shown in \cite{CDWII} Sections 3.1 and 3.2 (see also \cite{CDF}
Section 1.2), a generating function of the type 
\begin{eqnarray*}
S_{\rhd}(P,\xi) & = & \langle\xi,Y+[X,Y]+\frac{1}{2}[X,[X,Y]]+\ldots\rangle\\
 & = & \langle\Phi(\xi),P\rangle+{\mathcal{O}}(P^{2})
\end{eqnarray*}
where $\Phi:{\mathfrak{h}}^{*}\to{\mathfrak{h}}^{*}\times{\mathfrak{h}}^{*}$,
$\Phi(\xi)=(0,\xi)$, yields a symplectic micromorphism 
\[
([L_{S}],\Phi):T^{*}{\mathfrak{h}}^{*}\otimes T^{*}{\mathfrak{h}}^{*}\to T^{*}{\mathfrak{h}}^{*}
\]
where 
\begin{eqnarray*}
L_{S} & = & \left\{ \left((X,\,\frac{\partial S_{\rhd}}{\partial X}),(Y,\,\frac{\partial S_{\rhd}}{\partial Y}),(\frac{\partial S_{\rhd}}{\partial\xi},\,\xi)\right)\,|\,\xi\in{\mathfrak{g}}^{*},\,\,\, X,Y\in{\mathfrak{g}}\,\right\} \\
 & = & \left\{ \left((X,\langle[X,Y],\xi\rangle),(Y,{\rm Ad}_{X}^{*}(\xi)),({\rm Ad}_{X}(Y),\xi)\right)\,|\,\xi\in{\mathfrak{g}}^{*},\,\,\, X,Y\in{\mathfrak{g}}\,\right\} 
\end{eqnarray*}
which one recognizes to be the cotangent lift of the map $(X,Y)\mapsto{\rm Ad}_{X}(Y)$.\end{proof}
\begin{rema}
If ${\mathfrak{g}}$ is a Lie algebra, then ${\rm Ad}:{\mathfrak{g}}\times{\mathfrak{g}}\to{\mathfrak{g}}$
is the adjoint action of the local/formal group $({\mathfrak{g}},BCH)$
on ${\mathfrak{g}}$ by Lemma \ref{formula_BCH_conjugation}. The
cotangent lift of this action is a Hamiltonian action of $({\mathfrak{g}},BCH)$
on $T^{*}{\mathfrak{g}}$, given by $T^{*}{\rm Ad}_{X}:T^{*}{\mathfrak{g}}\to T^{*}{\mathfrak{g}}$
for all $X\in{\mathfrak{g}}$. This Hamiltonian action has an equivariant
momentum map $J:T^{*}{\mathfrak{g}}\to{\mathfrak{g}}^{*}$ given by
$J(Y,\xi)\,=\,\langle\xi,{\rm ad}_{Y}\rangle$, i.e. 
\[
\langle X,J(Y,\xi)\rangle\,=\,\langle\xi,[Y,X]\rangle.
\]
Under the identification $T^{*}{\mathfrak{g}}\cong T^{*}{\mathfrak{g}}^{*}$
($\cong{\mathfrak{g}}\times{\mathfrak{g}}^{*}$), the cotangent lift
$T^{*}{\rm Ad}_{X}$ gives a Hamiltonian action of $({\mathfrak{g}},BCH)$
on $T^{*}{\mathfrak{g}}^{*}$. This yields an action of the (local)
symplectic groupoid $\xymatrix{T^{*}{\mathfrak{g}}^{*}\ar@<2pt>[r]^{s}\ar@<-2pt>[r]_{t} & {\mathfrak{g}}^{*}}
$ on $J:T^{*}{\mathfrak{g}}^{*}\to{\mathfrak{g}}^{*}$ whose graph
\[
\rho_{{\rm Ad}}:T^{*}{\mathfrak{g}}^{*}\,_{s}\times_{J}T^{*}{\mathfrak{g}}^{*}\to T^{*}{\mathfrak{g}}^{*}
\]
is a (germ of a) Lagrangian submanifold yielding the symplectic micromorphism
(as explained in \cite{CDWII}) 
\[
T^{*}{\rm Ad}\,=\,\Big\{\Big(\big(X,\, J({\rm Ad}_{-X}(Y),{\rm Ad}_{X}^{*}(\xi)\big),\,\big(Y,\xi\big),\,\big({\rm Ad}_{-X}(Y),{\rm Ad}_{X}^{*}(\xi)\big)\Big):\, X,Y,\xi\Big\},
\]
which we can simplify using the equivariance of the moment map $J$:
\begin{eqnarray*}
\langle X,J({\rm Ad}_{-X}(Y),{\rm Ad}_{X}^{*}(\xi))\rangle & = & \langle X,J(T^{*}{\rm Ad}_{X}(Y,\xi))\rangle\\
 & = & \langle X,{\rm Ad}_{X}^{*}J(Y,\xi)\rangle\\
 & = & \langle{\rm Ad}_{X}X,J(Y,\xi)\rangle\\
 & = & \langle X,J(Y,\xi)\rangle,
\end{eqnarray*}
where we have used that $[X,X]=0$ in the Lie algebra ${\mathfrak{g}}$.
Therefore, we obtain
\[
T^{*}{\rm Ad}\,=\,\Big\{\Big((X,\, J(Y,\xi)),\,(Y,\xi),\, T^{*}{\rm Ad}_{X}(Y,\xi)\Big):\,(X,Y,\xi)\in T^{*}\mathfrak{g}^{*}\oplus T^{*}\mathfrak{g}^{*}\}.
\]

Under the identification $T^{*}{\mathfrak{g}}\cong T^{*}{\mathfrak{g}}^{*}\cong{\mathfrak{g}}\times{\mathfrak{g}}^{*}$,
we have that
\[
T^{*}{\rm Ad}_{X}=T^{*}{\rm Ad}_{X}^{*},
\]
i.e. the cotangent lift of the adjoint action and that of the coadjoint action
coincide. Thus quantizing $\rhd:{\mathfrak{h}}\times{\mathfrak{h}}\to{\mathfrak{h}}$
should be the same as quantizing the coadjoint action ${\rm Ad}_{X}^{*}:{\mathfrak{h}}^{*}\to{\mathfrak{h}}^{*}$.
Observe that switching to Leibniz algebras, the adjoint action ${\rm Ad}_{X}$
becomes a left rack action in the sense of Definition 
\ref{definition_rack_action}. Therefore the coadjoint
action ${\rm Ad}_{X}^{*}$ becomes naturally a left rack action on
${\mathfrak{h}}^{*}$ via the formula 
\[
({\rm Ad}_{X}^{*}(f))(Y)\,:=\, f((X\rhd-)^{-1} Y),
\]
see Lemma \ref{coadjoint_action}. Hence the stage is set to study the object which
should replace the symplectic groupoid 
$\xymatrix{T^{*}{\mathfrak{g}}^{*}\ar@<2pt>[r]^{s}\ar@<-2pt>[r]_{t} & {\mathfrak{g}}^{*}}$
in the context of deformation quantization of Leibniz algebras. We will do this in 
subsequent work. 
\end{rema} 


We are now ready to quantize $T^{*}\rhd:T^{*}{\mathfrak{h}}\otimes T^{*}{\mathfrak{h}}\to T^{*}{\mathfrak{h}}$.
As before, the family of semi-classical FIO quantizing the symplectic
micromorphism is given by 
\[
Q^{a}(T^{*}\rhd)(f\otimes g)(\xi)\,=\,\int_{{\mathfrak{g}}\times{\mathfrak{g}}}\widehat{f}(X)\widehat{g}(Y)a(X,Y,\xi)e^{\frac{i}{\hbar}S_{\rhd}(X,Y,\xi)}\frac{dXdY}{(2\pi\hbar)^{n}},
\]
where $a$ is the germ of an amplitude and $\widehat{f}$ and $\widehat{g}$
are the asymptotic Fourier transforms.
\begin{theo}
For $a=1$, the operation
\[
\rhd_{\hbar}:{\mathcal{C}}^{\infty}({\mathfrak{h}}^{*})[[\epsilon]]\otimes{\mathcal{C}}^{\infty}({\mathfrak{h}}^{*})[[\epsilon]]\to{\mathcal{C}}^{\infty}({\mathfrak{h}}^{*})[[\epsilon]]
\]
defined by
\[
f\rhd_{\hbar}g\,:=\, Q^{a=1}(T^{*}\rhd)(f\otimes g)
\]
is a quantum rack, i.e. 

(1) $\rhd_{\hbar}$ restricted to $U_{\mathfrak{h}}=\{E_{X}\,|\, X\in{\mathfrak{h}}\}$
is a rack structure and
\[
e^{\frac{i}{\hbar}X}\rhd_{\hbar}e^{\frac{i}{\hbar}Y}\,=\, e^{\frac{i}{\hbar}{\rm conj}_{*}(X,Y)},
\]

(2) $\rhd_{\hbar}$ restricted to $\rhd_{\hbar}:U_{\mathfrak{h}}\times{\mathcal{C}}^{\infty}({\mathfrak{h}}^{*})\to{\mathcal{C}}^{\infty}({\mathfrak{h}}^{*})$
is a rack action and 
\[
(e^{\frac{i}{\hbar}X}\rhd_{\hbar}f)(\xi)\,=\,({\rm Ad}_{-X}^{*}f)(\xi).
\]

\end{theo}
Moreover, $\rhd_{\hbar}$ coincides with the Gutt quantum rack $f\rhd_{a}g:=f*_{a}g*_{a}\overline{f}$
on the restrictions in the Lie case (although it is different on the
whole ${\mathcal{C}}^{\infty}({\mathfrak{h}}^{*})[[\epsilon]]$).
\begin{rema}
Actually, Property (2) in the theorem above holds also for square
integrable functions, and we even obtain a unitary rack action: 
\[
\rhd_{\hbar}:U_{\mathfrak{h}}\times L^{2}({\mathfrak{h}}^{*})\to L^{2}({\mathfrak{h}}^{*}).
\]
\end{rema}
\begin{proof}
The first property follows from the fact that exponentials Fourier
transform to delta functions: 
\begin{eqnarray*}
\left(e^{\frac{i}{\hbar}\bar{X}}\rhd_{\hbar}e^{\frac{i}{\hbar}\bar{Y}}\right)(\xi) & = & \int\widehat{e^{\frac{i}{\hbar}\bar{X}}}\widehat{e^{\frac{i}{\hbar}\bar{Y}}}e^{\frac{i}{\hbar}\langle\xi,{\rm Ad}_{X}(Y)\rangle}\frac{dXdY}{(2\pi\hbar)^{{\rm dim}({\mathfrak{h}})}}\\
 & = & (2\pi\hbar)^{{\rm dim}({\mathfrak{h}})}\int\delta_{\bar{X}}(X)\delta_{\bar{Y}}(Y)e^{\frac{i}{\hbar}\langle\xi,{\rm Ad}_{X}(Y)\rangle}\frac{dXdY}{(2\pi\hbar)^{{\rm dim}({\mathfrak{h}})}}\\
 & = & e^{\frac{i}{\hbar}\langle{\rm Ad}_{\bar{X}}(\bar{Y}),\xi\rangle}\,=\, e^{\frac{i}{\hbar}\langle{\rm conj}_{*}(\bar{X},\bar{Y}),\xi\rangle}.
\end{eqnarray*}
Now $\rhd_{\hbar}$ satisfies the rack identity on $U_{\mathfrak{h}}$,
because ${\rm conj}_{*}$ does. Furthermore, 
\[
E_{Y}\mapsto E_{X}\rhd_{\hbar}E_{Y}=E_{{\rm conj}_{*}(X,Y)}
\]
is bijective for all $X\in{\mathfrak{h}}$, because $Y\mapsto{\rm conj}_{*}(X,Y)$
is. It is also clear from the formula above that this rack structure
coincides with the Gutt rack structure in the case of a Lie algebra.

The second property also follows from the fact that exponentials Fourier-transform
to delta functions: 
\begin{eqnarray*}
\left(e^{\frac{i}{\hbar}\bar{X}}\rhd_{\hbar}f\right)(\xi) & = & \int\widehat{e^{\frac{i}{\hbar}\bar{X}}}\widehat{f}(Y)e^{\frac{i}{\hbar}\langle\xi,{\rm Ad}_{X}(Y)\rangle}\frac{dXdY}{(2\pi\hbar)^{{\rm dim}({\mathfrak{h}})}}\\
 & = & (2\pi\hbar)^{({\rm dim}({\mathfrak{h}}))/2}\int\delta_{\bar{X}}(X)\widehat{f}(Y)e^{\frac{i}{\hbar}\langle\xi,{\rm Ad}_{X}(Y)\rangle}\frac{dXdY}{(2\pi\hbar)^{{\rm dim}({\mathfrak{h}})}}\\
 & = & \frac{1}{(2\pi\hbar)^{({\rm dim}({\mathfrak{h}}))/2}}\int\widehat{f}(Y)e^{\frac{i}{\hbar}\langle{\rm Ad}_{-X}^{*}\xi,Y\rangle}dY\\
 & = & f({\rm Ad}_{-X}^{*}\xi).
\end{eqnarray*}
One sees that this defines a rack action from the fact that the coadjoint
action ${\rm Ad}_{-X}^{*}$ is a rack action.
\end{proof}

Let us now show that the first term of the quantized bracket is indeed the bracket 
(\ref{Poisson_bracket}). For an oscillatory integral as the above expression for $f\rhd_{\hbar} g$,
there is a well defined procedure of expansion in terms of Feynman graphs, in case
the integral has a unique, non-degenerate critical point. This procedure is for example explained
in \cite{DheMen}.  

\begin{theo} \label{bracket_computation}
\begin{enumerate}
\item[(a)] The above oscillatory integral $f\rhd_{\hbar} g$ has a unique, non-degenerate critical point
and admits thus a Feynman expansion in terms of graphs. 

\item[(b)] The first term of the formal expansion of 
\begin{equation}  \label{exponential_formula}
(f\rhd_{\hbar} g)(\xi)\,=\,\int f(\bar{\xi})g(\bar{\eta})e^{\frac{i}{\hbar}(-\bar{X}\bar{\zeta}-
\bar{Y}\bar{\eta}+\langle\xi,\exp({\rm ad}_{\bar{X}})(\bar{Y})\rangle)}\frac{d\bar{X}d\bar{Y}d\bar{\zeta}
d\bar{\eta}}{(2\pi\hbar)^n}
\end{equation}
in powers of $\hbar$ is the Leibniz-Poisson bracket (\ref{Poisson_bracket}), i.e.
$$
\{f,g\}(\xi)\,=\,\langle\xi,[df(0),dg(\xi)]\rangle.
$$
\end{enumerate}
\end{theo}

\begin{proof}
Observe that in equation (\ref{exponential_formula}), we wrote out explicitely the 
asymptotic Fourier transforms of $f$ and $g$. The total phase of the above oscillatory 
integral is thus
$$S_{\xi}(\bar{X},\bar{Y},\bar{\zeta},\bar{\eta})\,=\,-\bar{X}\bar{\zeta}-
\bar{Y}\bar{\eta}+\langle\xi,\exp({\rm ad}_{\bar{X}})(\bar{Y})\rangle.$$
The phase $S_{\xi}(\bar{X},\bar{Y},\bar{\zeta},\bar{\eta})$ has 
$$c_{\xi}\,=\,(\bar{X}=0,\bar{Y}=0,\bar{\zeta}=0,\bar{\eta}=\xi)$$
as its unique critical point. This means that for any given $\xi$, 
$c_{\xi}$ is unique within the points $c:=(\bar{X},\bar{Y},\bar{\zeta},\bar{\eta})$ such that
$$\frac{\partial S_{\xi}}{\partial \bar{X}}(c)\,=\,0,\,\,\,\,
\frac{\partial S_{\xi}}{\partial \bar{Y}}(c)\,=\,0,\,\,\,\,
\frac{\partial S_{\xi}}{\partial \bar{\zeta}}(c)\,=\,0,\,\,\,\,
\frac{\partial S_{\xi}}{\partial \bar{\eta}}(c)\,=\,0.$$
The critical point $c_{\xi}$ is easily computed from the partial derivatives. It turns
out that
$$\frac{\partial S_{\xi}}{\partial \bar{X}}(c)\,=\,-\bar{\zeta}+ T_1,\,\,\,\,
\frac{\partial S_{\xi}}{\partial \bar{Y}}(c)\,=\,-\bar{\eta}+T_2,\,\,\,\,
\frac{\partial S_{\xi}}{\partial \bar{\zeta}}(c)\,=\,-\bar{Y},\,\,\,\,
\frac{\partial S_{\xi}}{\partial \bar{\eta}}(c)\,=\,-\bar{X},$$
where the term $T_1$ is the derivative of $\bar{X}\mapsto\langle\xi,\exp({\rm ad}_{\bar{X}})(\bar{Y})\rangle$
and the term $T_2$ is the derivative of $\bar{Y}\mapsto\langle\xi,\exp({\rm ad}_{\bar{X}})(\bar{Y})\rangle$.
One concludes from setting the third and fourth equation equal to zero that $\bar{X}=\bar{Y}=0$. 
The first term of $Y\mapsto\langle\xi,\exp({\rm ad}_{\bar{X}})(\bar{Y})\rangle$ is $\xi\bar{Y}$, thus
the constant term in $T_2$ is $\xi$. All other terms in $T_1$ and $T_2$ are zero 
at the critical point due to $\bar{X}=\bar{Y}=0$. In conclusion 
$c_{\xi}\,=\,(\bar{X}=0,\bar{Y}=0,\bar{\zeta}=0,\bar{\eta}=\xi)$.   

The Hessian of $S_{\xi}$ at the critical point $c_{\xi}$ reads in block notation
$$D^2S_{\xi}(c_{\xi})\,=\,\left(
\begin{matrix} 0 & c_{ij}^k\xi_k & -1 & 0 \\
c_{ij}^k\xi_k & 0 & 0 & -1 \\
-1 & 0 & 0 & 0 \\
0 & -1 & 0 & 0 \end{matrix}\right),$$
where $c_{ij}^k$ are the structure constants of the Leibniz algebra ${\mathfrak h}$
and in Einstein convention, the sum over repeated indices is understood. 

Denoting the matrix $D^2S_{\xi}(c_{\xi})$ simply by $B$, 
it is evident that $\det(B)=1$, thus the critical point $c_{\xi}$
is non-degenerate. Moreover, the signature of $B$ is $0$. The Feynman expansion
(cf \cite{DheMen}) therefore reads
\begin{eqnarray*}
I(\hbar)&=&(f\rhd_{\hbar} g)(\xi)=\\
&=&\frac{e^{\frac{i\pi}{4}{\rm sign}(B)}}{\sqrt{|\det(B)|}}\big(\sum_{\Gamma\in G_{3\geq}(2)}
\frac{(i\hbar)^{|{\rm E}_{\Gamma}|-|{\rm V}_{\Gamma}^{\rm int}|}}{|{\rm Aut}(\Gamma)|}
F_{\Gamma}(S_{\xi};f,g)\big)\\
&=&e^{\frac{i\pi}{4}}\big(\sum_{\Gamma\in G_{3\geq}(2)}
\frac{(i\hbar)^{|{\rm E}_{\Gamma}|-|{\rm V}_{\Gamma}^{\rm int}|}}{|{\rm Aut}(\Gamma)|}
F_{\Gamma}(S_{\xi};f,g)\big).
\end{eqnarray*}
These sums are sums over the set $G_{3\geq}(2)$ of Feynman graphs $\Gamma$ with $2$ external vertices
and internal vertices of valence greater or equal to $3$. For the definition of a Feynman graph, we
refer the reader to \cite{DheMen}.  
$|{\rm E}_{\Gamma}|$ is the cardinality of the set of edges of $\Gamma$, ${\rm V}_{\Gamma}^{\rm int}$
is the set of internal vertices of $\Gamma$. ${\rm Aut}(\Gamma)$ is the number of symmetries 
of $\Gamma$. 
To each $\Gamma$, one associates an amplitude $F_{\Gamma}(S_{\xi};f,g)$ in a way which
is specified in {\it loc. cit.}. Namely, $F_{\Gamma}(S_{\xi};f,g)$ is a product of two partial
derivatives of $S_{\xi}$ (represented by the internal vertices) and partial derivatives of $f$ and $g$
(represented by the external vertices) all of which are evaluated at the critical point
$c_{\xi}$ and contracted using the matrix $B^{-1}$. 

The first terms of the expansion of (\ref{exponential_formula}) in powers of $\hbar$ read therefore    
$$(f\rhd_{\hbar} g)(\xi)\,=\,f(0)g(\xi)+\frac{i}{\hbar}\{f,g\}(\xi)+{\mathcal O}(\hbar),$$
where 
$$\{f,g\}(\xi)\,=\,-\sum_{i,j,k}c_{i,j}^k\frac{\partial f}{\partial \xi_i}(0)
\frac{\partial g}{\partial \xi_j}(\xi)\xi_k,$$
as in formula (\ref{Poisson_bracket}).
\end{proof}

\begin{rem}
\begin{enumerate}
\item[(a)] It is rather straight forward to compute the terms in this starproduct, the graphs which we have 
to consider are rather easy. For example, there are no inner loops.  
\item[(b)] The zeroth term of the expansion, i.e. the product $f\otimes g\mapsto f(0)g(\xi)$, is actually 
associative.
\end{enumerate}
\end{rem}


\begin{thebibliography}{50}

\bibitem{BFFLS} Bayen, F.; Flato, M.; Fronsdal, C.; Lichnerowicz,
A.; Sternheimer, D. \textit{Deformation theory and quantization. I.
Deformations of symplectic structures.} Ann. Physics \textbf{111}
(1978), no. 1, 61--110

\bibitem{ADW} Albert, C. ; Dazord, P. \textit{Th\'eorie des groupo\"{\i}des
symplectiques.} Chapitre II. Groupo\"{\i}des symplectiques. Publications
du D\'epartement de Math\'ematiques. Nouvelle s\'erie, 27--99, Publ. D\'ep.
Math. Nouvelle S\'er., 1990, Univ. Claude-Bernard, Lyon, 1990 

\bibitem{BenAmar1} Ben Amar, N. \textit{K-star products on dual of
Lie algebras}. J. Lie Theory {\bf 13} (2003) 329--357

\bibitem{BenAmar2} Ben Amar, N. \textit{A comparison between Rieffel's
and Kontsevich's deformation quantizations for linear Poisson tensors.}
Pacific J. Math. {\bf 229} (2007), no. 1, 1--24

\bibitem{Bou} Bourbaki, N. \textit{Lie groups and Lie algebras. Chapters
1--3.} Reprint of the 1989 English translation. Elements of Mathematics
(Berlin). Springer-Verlag, Berlin, 1998

\bibitem{Can} Canez, Santiago Valencia \textit{Double Groupoids,
Orbifolds, and the Symplectic Category.} Thesis (Ph.D.), University
of California, Berkeley, 2011

\bibitem{CDF} Cattaneo, Alberto S.; Dherin, Benoit; Felder, Giovanni
\textit{Formal symplectic groupoid.} Comm. Math. Phys. \textbf{253}
(2005), no. 3, 645--674

\bibitem{CDWI} Cattaneo, Alberto S.; Dherin, Benoit; Weinstein, Alan
\textit{Symplectic microgeometry I: micromorphisms.} J. Symplectic
Geom. \textbf{8} (2010), no. 2, 205--223

\bibitem{CDWII} Cattaneo, Alberto S.; Dherin, Benoit; Weinstein,
Alan \textit{Symplectic microgeometry II: generating functions. } Bull.
Braz. Math. Soc. (N.S.) \textbf{42} (2011), no. 4, 507--536

\bibitem{CDWIII} Cattaneo, Alberto S.; Dherin, Benoit; Weinstein,
Alan \textit{Symplectic microgeometry III: monoids.}
J. Symplectic Geom. {\bf 11} (2013), no. 3, 319--341

\bibitem{CDWIV} Cattaneo, Alberto S.; Dherin, Benoit; Weinstein,
Alan \textit{Symplectic microgeometry IV: quantization} (in preparation)

\bibitem{Cov} Covez, Simon \textit{L'intégration locale des algèbres
de Leibniz.} PhD Thesis, Nantes 2010 (see also his article on the
same subject: {\it The local integration of Leibniz algebras.} 
Ann. Inst. Fourier (Grenoble) {\bf 63} (2013), no. 1, 1--35.)

\bibitem{DheMen} Dherin, Benoit; Mencattini, Igor 
\textit{Quantizations of Momentum Maps and G-Systems}
{\tt arXiv:1212.6489} 

\bibitem{FenRou} Fenn, Roger; Rourke, Colin \textit{Racks and links
in codimension two.} J. Knot Theory Ramifications 1 (1992), no. 4,
343--406

\bibitem{GraMar} Grabowski, Janusz; Marmo, Giuseppe
\textit{Non-antisymmetric versions of Nambu-Poisson and algebroid brackets.} 
J. Phys. A {\bf 34} (2001), no. 18, 3803--3809

\bibitem{Gut} Gutt, Simone \textit{An explicit $\ast$--product on
the cotangent bundle of a Lie group.} Lett. Math. Phys. \textbf{7}
(1983), no. 3, 249--258

\bibitem{Kin} Kinyon, Michael \textit{Leibniz algebras, Lie racks,
and digroups.} J. Lie Theory \textbf{17} (2007) no. 1, 99--114

\bibitem{KinWei} Kinyon, Michael; Weinstein, Alan \textit{Leibniz
algebras, Courant algebroids, and multiplications on reductive homogeneous
spaces.} Amer. J. Math. \textbf{123} (2001) no. 3, 525--550

\bibitem{Kon} Kontsevich, Maxim \textit{Deformation quantization
of Poisson manifolds.} Lett. Math. Phys. \textbf{66} (2003), no. 3,
157--216

\bibitem{KHN} Neeb, Karl-Hermann \textit{Central extensions of infinite-dimensional
Lie groups.} Ann. Inst. Fourier (Grenoble) \textbf{52} (2002), no.
5, 1365--1442

\bibitem{LibSev} Li-Bland, David; Severa, Pavol \textit{Integration
of Exact Courant Algebroids} 
Electron. Res. Announc. Math. Sci. {\bf 19} (2012) 58--76

\bibitem{Roy} Roytenberg, Dmitry \textit{On weak Lie 2-algebras.}
XXVI Workshop on Geometrical Methods in Physics, 180--198, AIP Conf.
Proc., 956, Amer. Inst. Phys., Melville, NY, 2007

\bibitem{Tuy} Tuynman, Ghys \textit{A proof of Lie's third theorem.}
Pub. IRMA Lille Vol. {\bf 34}, no. X, 1994.

\bibitem{Uch} Uchino, Kyousuke
\textit{Noncommutative Poisson brackets on Loday algebras and related deformation quantization.} 
J. Symplectic Geom. {\bf 11} (2013), no. 1, 93--108

\bibitem{Var} Varadarajan, V.S. \textit{Lie Groups, Lie
Algebras and Their Representations}, Springer GTM {\bf 102},
Springer New York 1974

\bibitem{Wei} Weinstein, Alan \textit{Symplectic categories.} Port.
Math. \textbf{67} (2010), no. 2, 261--278

\bibitem{WeiII} Weinstein, Alan \textit{The symplectic category}.
\textit{Differential Geometric Methods in Mathematical Physics} (Clausthal,
1980), pp. 45--51, Lecture Notes in Math., 905, Springer, 1982. 

\bibitem{WeiIII} Weinstein, Alan . Noncommutative geometry and geometric
quantization. \textit{Symplectic Geometry and Mathematical Physics}
(Aix-en-Provence, 1990), 446--461, Progr. Math. {\bf 99}, 1991.

\bibitem{Zak} Zakrzewski, Stanislaw \textit{Quantum and classical
pseudogroups I and II.} Comm. Math. Phys. \textbf{134} (1990), 347--370

\end{thebibliography}
\end{document}